\documentclass[letterpaper,11pt,reqno]{amsart} 
\usepackage[portrait,margin=1in]{geometry}
\usepackage{systeme,mathrsfs,xfrac} 
\usepackage[colorlinks=true,linkcolor=blue,citecolor=blue,urlcolor=blue,pdfencoding=auto, psdextra]{hyperref} 
\usepackage{amsmath,amssymb,amsthm,amsfonts,amsbsy,latexsym,dsfont,color} 
\usepackage[foot]{amsaddr}
\usepackage{comment}
\usepackage{tikz}
\usetikzlibrary{positioning,arrows.meta,calc,bending}

\usepackage[textsize=tiny]{todonotes}
\usepackage{regexpatch}
\makeatletter
\xpatchcmd{\@todo}{\setkeys{todonotes}{#1}}{\setkeys{todonotes}{inline,#1}}{}{}
\makeatother

\usepackage{enumerate}
\newenvironment{enumeratei}{\begin{enumerate}[\upshape i.]}{\end{enumerate}}
\newenvironment{enumeratea}{\begin{enumerate}[\upshape a)]}{\end{enumerate}}
\newenvironment{enumeraten}{\begin{enumerate}[\upshape 1.]}{\end{enumerate}}

\newtheorem{thm}{Theorem}[section]
\newtheorem{lem}[thm]{Lemma}
\newtheorem{cor}[thm]{Corollary}
\newtheorem{prop}[thm]{Proposition}
\newtheorem{defn}[thm]{Definition}
\newtheorem{rem}[thm]{Remark}

\renewcommand{\le}{\leqslant} 
\renewcommand{\ge}{\geqslant} 
\renewcommand{\leq}{\leqslant} 
 
\newcommand{\eset}{\varnothing}

\newcommand{\ind}{\mathds{1}}
\newcommand{\eps}{\varepsilon}

\newcommand{\norm}[1]{\left\Vert#1\right\Vert}
\newcommand{\abs}[1]{\left\vert#1\right\vert}

\newcommand{\ie}{\emph{i.e.,}}

\let\ga=\alpha  \let\gc=\gamma \let\gd=\delta 
    \let\gk=\kappa \let\gl=\lambda        \let\go=\omega   \let\gs=\sigma

\newcommand{\cC}{\mathcal{C}}
\newcommand{\cE}{\mathcal{E}}

\newcommand{\cL}{\mathcal{L}}

\newcommand{\cP}{\mathcal{P}}

\newcommand{\cW}{\mathcal{W}}
  
\newcommand{\vB}{\mathbf{B}}

\newcommand{\mvzero}{\boldsymbol{0}}

\newcommand{\mvJ}{\boldsymbol{J}}\newcommand{\mvL}{\boldsymbol{L}}

\newcommand{\mvS}{\boldsymbol{S}}

\newcommand{\mvp}{\boldsymbol{p}}

\newcommand{\mvgo}{\boldsymbol{\omega}}

\newcommand{\mvmu}{\boldsymbol{\mu}}

\newcommand{\fL}{\mathfrak{L}}

\newcommand{\fR}{\mathfrak{R}}

\newcommand{\fg}{\mathfrak{g}}

\newcommand{\fs}{\mathfrak{s}}

\newcommand{\dN}{\mathds{N}}

\newcommand{\dR}{\mathds{R}}

\newcommand{\dZ}{\mathds{Z}} 
\DeclareMathOperator{\E}{\mathds{E}}
\DeclareMathOperator{\pr}{\mathds{P}}

\DeclareMathOperator{\var}{Var}
\DeclareMathOperator{\cov}{Cov}

\DeclareMathOperator{\N}{N}
\newcommand{\lc}{\ell_{\mathsf{c}}}
\newcommand{\sfH}{\mathsf{H}}
\newcommand{\sfP}{\mathsf{P}}
\newcommand{\sfT}{\mathsf{T}}
\newcommand{\sfR}{\mathsf{R}}
\newcommand{\sfJ}{\mathsf{J}}
\newcommand{\sfRT}{\mathsf{R}^{\mathsf{T}}}
\newcommand{\sfRH}{\mathsf{R}^{\mathsf{H}}}
\newcommand{\sfN}{\mathsf{N}}
\DeclareMathOperator{\Env}{\mathrm{Env}}
\newcommand{\supp}{\operatorname{supp}}

\begin{document}
\title[FPP on Spread-out line graph]{First-Passage Percolation on Spread-out line graphs: Microscopic Regime}
\author[Dey]{Partha S.~Dey$^\star$}
\author[Kim]{Daecheol Kim$^\dagger$}
\address{Department of Mathematics, University of Illinois Urbana--Champaign, 1409 W Green Street, Urbana, Illinois 61801}
\email{$^\star$psdey@illinois.edu, $^\dagger$dk43@illinois.edu}
\date{\today}

%%%%%%%%%%%%%%%%%%%%%%%%%%%%%%%%%%%%%%%%%%%%%%%%%%%%%%%%%%%%%%%%%%%%%%%%%
\begin{abstract}
    We study first-passage percolation on the $\ell$-spread-out line graph, where each vertex $i\in\{0,1,\dots,n\}$ is connected to all others at distance at most $\ell$. Here, we focus on the {\it microscopic} regime, with $\ell$ fixed as $n\to\infty$. Independent nonnegative weights are assigned to these edges.  We obtain a law of large numbers and precise fluctuation results for the passage time $T_n$ from $0$ to $n$.  If the weight distribution has finite variance or a heavy tail with exponent above $2/\lc$ where $\lc:=\ell(\ell+1)/2$, then $T_n$ satisfies a Gaussian CLT with $\sqrt{n}$ scaling.  In contrast, for heavier-tailed distributions, with index below the threshold, we show that $T_n$, appropriately centered and scaled, converges to a non-Gaussian stable law.  We also prove an LLN and CLT for the number of edges in the minimizing path.  The key tool is a \emph{pivot-node} decomposition; the geodesic can be segmented into i.i.d.~blocks, leading to a renewal structure. Our results extend the classical one-dimensional CLT to include finite-range connectivity and heavy tails, revealing a new distributional phase transition in the fluctuations of $T_n$. 
\end{abstract}

%%%%%%%%%%%%%%%%%%%%%%%%%%%%%%%%%%%%%%%%%%%%%%%%%%%%%%%%%%%%%%%%%%%%%%%%%
\maketitle
\setcounter{tocdepth}{1}\tableofcontents

%%%%%%%%%%%%%%%%%%%%%%%%%%%%%%%%%%%%%%%%%%%%%%%%%%%%%%%%%%%%%%%%%%%%%%%%%
\section{Introduction}\label{sec:intro}

First-passage percolation or FPP is a fundamental probabilistic model for studying the spread of information, transport phenomena, and growth processes in random environments. Originally introduced by Hammersley and Welsh~\cite{hamwelsh65} in the '60s to model fluid flow through porous media, it has since become a central object in probability theory, random geometry, statistical physics, and network science. The model provides a mathematical framework for understanding how randomness in an underlying medium influences the speed, structure, and fluctuations of propagation processes.

In its classical form, FPP assigns independent non-negative random weights to the edges of a graph, where each weight represents the cost, time, or resistance associated with traversing that edge. The first-passage time between two vertices is defined as the minimum total weight among all connecting paths and describes the fastest possible route through the random environment. Beyond its mathematical significance, FPP has found applications across diverse disciplines, including the study of transport in disordered materials, information dissemination in communication networks, epidemic spread, financial markets, and biological systems. We refer to the survey~\cite{adh17} for more details.

A central objective in FPP is to characterize the asymptotic behavior of passage times as the underlying graph becomes large. Classical questions concern the growth rate of passage times, the nature of their fluctuations, and the emergence of phase transitions driven by changes in the geometry of the network or the distribution of edge weights. Understanding how network structure influences propagation speed remains a fundamental challenge. 
Our work investigates this question in a setting where additional resources permit information to travel over longer distances at equivalent costs. Specifically, we seek to understand: at what spatial scale do long-range connections substantially alter propagation speed, and how does the qualitative behavior of optimal transmission paths change across this transition?

The classical lattice setting imposes a fixed local geometry. Many spatial networks arising in transportation, infrastructure, communication, and biological systems are local at short scales but not purely nearest-neighbor; distance and wiring costs constrain the network, while additional local or longer-range connections may substantially change global transport properties~\cite{barthelemy11}. This trade-off is also central in small-world models, where a locally clustered graph is perturbed by non-local links and the resulting network can retain local structure while dramatically reducing typical graph distances and changing spreading behavior~\cite{wattsstrogatz98,newmanwatts99,kleinberg00}. At the level of growing real-world networks, densification phenomena in social, information, and urban networks show that the effective number of contacts or links may increase with the system size~\cite{leskovec07,schlapfer14,shutters18}. These examples, together with the multiple growth regimes known for the long-range FPP with distance-dependent passage times~\cite{chatterjeeDey2016LRFPP}, motivate spatial graph families in which the connection range is a tunable parameter rather than a fixed background feature.

The spread-out graph is a deterministic and analytically tractable model of this idea. In particular, we consider a family of spread-out graphs obtained from a locally finite base graph $G$. Given an integer $\ell \ge 2$, we construct the $\ell$-spread-out graph $G^{(\ell)}$ by connecting every pair of vertices whose graph distance is at most $\ell$. The parameter $\ell$ separates the spatial geometry from the weight distribution: when $\ell$ is fixed, the graph geometry dominates at large scales but has several local bypasses through each cut; when $\ell$ grows with $n$, the number of available local routes increases and the model moves toward a highly connected, mean-field-like regime, when $\ell$ is comparable to the graph diameter.

We study first-passage percolation on $G^{(\ell)}$ with independent and identically distributed non-negative edge weights. Let 
\[
T^{(\ell)}(u,v):=\text{the first-passage time from vertex $u$ to vertex $v$},
\]
representing the minimum time required for information to propagate across the network from vertex $u$ to vertex $v$. Our goal is to determine how the asymptotic, fluctuation, and distributional properties of $T^{(\ell)}(u,v)$ evolve when the graph distance between $u$ and $v$ increases to infinity and the connectivity parameter $\ell$ varies. Here, we focus on three interrelated questions, namely, (a)~Propagation Speed, (b)~Distributional Phase Transitions, and (c)~Optimal Path Geometry.
%\begin{enumeratei}
%\item \textbf{Propagation Speed}.  How does the expected first-passage time scale with the graph distance, and how does this scaling depend on $\ell$?
%\item \textbf{Distributional Phase Transitions}. Classical central limit behavior is expected when $\ell$ is small relative to $n$, while qualitatively different, non-Gaussian fluctuations may emerge when $\ell$ is of order $n$. We aim to identify and characterize the transition between these regimes.
%\item \textbf{Optimal Path Geometry}. For small values of $\ell$, efficient paths are expected to follow a predominantly forward-moving strategy. As $\ell$ increases, however, optimal routes may exhibit increasingly complex local optimization patterns. We seek to understand how the geometry of geodesics changes across connectivity scales.
%\end{enumeratei}

In this article, we mainly focus on the case when $\ell\ge 2$ is fixed, and the base graph $G$ is the infinite line graph $\dZ$. We aim to study the scaling and distributional behavior of $T^{(\ell)}(0,n)$ as $n\to\infty$.
In general, one can consider three regimes for $\ell$, namely 
\begin{enumeratei}
	\item $\ell$ is a fixed integer,
	\item $1 \ll \ell \ll n,$
	\item $\ell = \gl  n \cdot (1+o(1))$, where $\gl>0$ is fixed.  
\end{enumeratei}

We refer to these regimes as the \textbf{Microscopic}, \textbf{Mesoscopic}, and \textbf{Macroscopic} regimes, respectively. 
This article treats the microscopic regime, where $\ell$ is fixed as $n\to\infty$. The mesoscopic and macroscopic regimes are treated in the companion paper~\cite{DK26b}, where the mean-field structure plays a nontrivial role. However, analysis of the microscopic regime in this article is self-contained and can be read independently of the companion paper~\cite{DK26b}.

The microscopic regime is important for two reasons.  First, it is the finite-range spatial baseline for the broader transition from local one-dimensional geometry to highly connected geometry. Second, even for fixed $\ell$, the model is not merely a trivial finite-range perturbation of nearest-neighbor one-dimensional FPP. The local spread-out geometry creates finitely many routes, and this structure changes the moment thresholds of passage times. As a result, the microscopic model exhibits a sharp dichotomy between Gaussian and non-Gaussian stable fluctuations, governed jointly by the local geometry and the tail of the edge-weight distribution.
%%%%%%%%%%%%%%%%%%%%%%%%%%%%%%%%%%%%%%%%%%%%%%%%%%%%%%%%%%%%%%%%%%%%%%%%%%%%%%%%%
\subsection{Main results}\label{ssec:results}

Consider the one-dimensional integer lattice $\dZ$ with edge set  $ \cE$. Given a positive integer $\ell$, the $\ell$-spread-out version (also known as $\ell$-th power) of $\dZ$ is defined by $\dZ^{(\ell)} = (\dZ, \cE^{(\ell)})$, where $ \cE^{(\ell)} $ is the set of edges such that
$
	\cE^{(\ell)} := \{ (x,y): x,y\in\dZ, 1\le y-x \leq \ell \}.
$
Let $\go_{x,y}=\go_{y,x}$ denote non-negative, i.i.d.~random weights defined for $(x,y)\in \cE^{(\ell)}.$ 
The \textit{first passage time} between vertices $0$ and $n$ is defined as
\begin{align}
    T^{(\ell)}(0,n)=\text{the minimal total weight of a path from $0$ to $n$}.
\end{align}

The microscopic regime corresponds to the case where $\ell$ is finite and independent of $n$. In this setting, the graph has a fixed local structure, as each vertex is connected to at most $\ell$ neighboring vertices on either side. The first-passage time $T^{(\ell)}(0,n)$ is heavily influenced by the local behavior of weights and paths.
An important quantity in this regime is $\lc$, defined by
\begin{align}\label{def:ell_c}
	\lc := \frac12 \ell(\ell+1).
\end{align}
The constant $\lc$ represents the number of edges crossing a unit-length edge $(u,u+1)$ in the $\ell$-spread-out line graph. It is also the maximum number of edge-disjoint paths between the two endpoint windows (see Lemma~\ref{lem:path_structure}). It plays a crucial role in determining the thresholds for the $k$-th moment of the renewal block first passage time between two appropriately defined pivot vertices. See Proposition~\ref{prop:k-moment}.

We assume that $\go$ satisfies one of the following assumptions 
\begin{enumerate}[{A}.1.]
    \item~\label{ass:A1} $\go$ has finite second moment, \ie\ $\E\go^2<\infty$,
    \item~\label{ass:A2} $\go$
    follows a heavy-tailed distribution given by
\begin{align}\label{eq:ass_go_micro}
	\pr(\go>x) = x^{-\gc}L(x) \text{ as } x\to\infty,
\end{align}
for some $\gc\in(0,2)$ and a slowly varying function $L(\cdot)$.
\end{enumerate} 
For simplicity, we define $\gc=2$ when $\go$ has a finite second moment. We now present the main results for the microscopic regime. 

\begin{thm}[Microscopic fluctuations]\label{thm:micro}
Assume that $\ell\ge2$ is fixed and $\go$ is not deterministic.
\begin{enumeratei}
	\item\label{thm:mean} If $\E\go<\infty$ or $\go$ satisfies condition~A.\ref{ass:A2} with $\gc>1/{\lc}$, then the weak law of large numbers holds, \ie\
\begin{align*}
    \frac{1}{n}T^{(\ell)}(0,n)
    \stackrel{\pr}{\to}
    \mu_{\mathrm{FPP}} \text{ as } n\to\infty
\end{align*}
for some constant $\mu_{\mathrm{FPP}}\in (0,\infty).$

\item\label{thm:clt} If $\E\go^2<\infty$ or $\go$ satisfies condition~A.\ref{ass:A2} with $\gc>2/{\lc}$, then the following Gaussian central limit theorem holds
\begin{align*}
    n^{-1/2}\cdot (T^{(\ell)}(0,n)-n\cdot \mu_{\mathrm{FPP}})
    \Longrightarrow
    \N(0,\gs_{\mathrm{FPP}}^2),
\end{align*}
where $\mu_{\mathrm{FPP}}$ is as in~\eqref{thm:mean}, and
$\gs_{\mathrm{FPP}}^2$ is a positive constant given in~\eqref{eq:var-CLT-def}.

\item\label{thm:sl} Finally, if condition~A.\ref{ass:A2} holds with $0<\gc<2/{\lc}$, then
\begin{align*}
    \frac1{a_n}\cdot (T^{(\ell)}(0,n)-b_n)
    \Longrightarrow
    Z_\alpha,
\end{align*}
where $\alpha=\gc\lc$ and $Z_\alpha$ is the totally right-skewed $\alpha$-stable law associated with the normalization $a_n=n^{1/(\gc \lc) +o(1)},b_n$ explicitly given in~\eqref{eq:stable-law-const}. 
\end{enumeratei}
\end{thm}

We also prove the Law of Large Numbers and Central Limit Theorem for the number of hop counts in the optimal path. A general result for additive functionals of the geodesic path is given in Corollary~\ref{cor:geo-stat}.

\begin{thm}[Microscopic hop-count LLN and CLT]\label{thm:micro-hop-clt}
Assume that $\ell\ge2$ is fixed and $\go$ is not deterministic. For integers $a<b$, let $H^{(\ell)}(a,b)$ denote the minimum number of edges among all geodesics from $a$ to $b$ in the $\ell$-spread-out line graph. Then
\begin{align}\label{eq:b2-hop-lln}
    \frac1n\cdot {H^{(\ell)}(0,n)}\to \mu_{\mathrm{Hop}}
    \text{ a.s.~and }
    \frac1n\cdot {\E H^{(\ell)}(0,n)}\to \mu_{\mathrm{Hop}}
\end{align}
for some constant $\mu_{\mathrm{Hop}}\in(0,\infty)$.
Moreover,
\begin{align}\label{eq:b2-hop-clt}
    n^{-1/2}\cdot (H^{(\ell)}(0,n)-n\cdot \mu_{\mathrm{Hop}})
    \Longrightarrow 
    \N(0,\gs^2_{\mathrm{Hop}}),
\end{align}
where $\gs^2_{\mathrm{Hop}}$ is a positive constant given in~\eqref{eq:hop-var-def}.
\end{thm}

\begin{rem}[Infinite line, finite intervals, and tori]\label{rem:finite-interval-torus}
The main results in this article are stated on the infinite $\ell$-spread-out line. The same argument applies to the finite interval $\{0,1,\ldots,n\}$ and the discrete $n$-cycle with the same limiting constants.
\end{rem}

For simplicity, we will omit the superscript $^{(\ell)}$ from the first-passage time and hopcount notations.
The main idea behind the proof of Theorem~\ref{thm:micro} and Theorem~\ref{thm:micro-hop-clt} is the identification of ``pivot'' nodes. One can view the pivot nodes as the set of all nodes $v$ belonging to all the geodesics starting from the left side of the node $v$ to the right side of $v$. These nodes allow us to decompose the first-passage time $T_n$ into independent contributions, simplifying the analysis. 
However, identifying this random set is nontrivial. To address this, we instead work with a simpler definition of pivot nodes based on local weight distribution as follows.
We will consider two mutually exclusive behaviors for the distribution of $\go$, namely,
\begin{enumerate}[{B}.1.]
\item~\label{ass:B1} There exists a positive real number $a>0$ such that 
\begin{equation}\label{eq:ass_pivot}
	\pr(\go > 2a) > 0 \text{ and } \pr(\go < a) > 0.
\end{equation}
    \item~\label{ass:B2} There exists a positive real number $c>0$ such that $\pr(\go\in [c,2c])=1$. Without loss of generality, we will assume that $c=1$ in this case and $1$ is in the support of $\go$.
\end{enumerate} 
Note that, if $\pr(\go \leq x) > 0$ for all $x > 0$ or $\go$ has unbounded support, then condition B.\ref{ass:B1} is satisfied.

\begin{lem}\label{lem:condB}
Let $\go$ be a nonnegative random variable which is not degenerate. Then conditions \textnormal{B.\ref{ass:B1}} and \textnormal{B.\ref{ass:B2}} are mutually exclusive, and at least one of them holds.
\end{lem}

\begin{proof}[Proof of Lemma~\ref{lem:condB}]
	Suppose that $\go$ does not satisfy condition~B.\ref{ass:B2}. Then,
	\begin{align}\label{eq:notB2}
		\pr(\go <b \text{ or } \go>2b) >0 \text{ for all } b>0.
	\end{align}
If $\pr(\go>2b)>0$ for all sufficiently large $b$, then we can find $a$ satisfying condition~B.\ref{ass:B1}. On the other hand, there exists $b_0>0$ such that $\pr(\go<b_0)=1$. Define the constants
\begin{align*}
	A:= \sup\{x:F(x) = 0 \} \text{ and } B:=\inf\{x:F(x)=1\}, \text{ where } F(x) = \pr(\go \le x).
\end{align*}
 By~\eqref{eq:notB2}, we have $2A <B$. Therefore, taking any $a\in (A,B/2)$ satisfies~B.\ref{ass:B1}.
\end{proof}

First, we define a notion of `pivot' nodes under the assumption~\textnormal{B.\ref{ass:B1}}.
\begin{defn}\label{defn:B1-pivot}
	Let $a>0$ be a fixed constant satisfying Assumption \textnormal{B.\ref{ass:B1}}.
A vertex $v$ is called a \emph{pivot} if
\begin{align}\label{def:pivot}
	\go_{x,v} < a, \;\; \go_{v,y} < a, \;\; \text{and} \;\; \go_{x,y} > 2a   \text{ for all }x<v<y \text{ with } 2\le  |x-y| \leq \ell.
\end{align}
\end{defn}
Under Condition~B.\ref{ass:B1}, Definition~\ref{defn:B1-pivot} gives an ordinary one-vertex pivot. Figure~\ref{fig:pivot-local} illustrates the local geometry imposed by the pivot condition.  The next Lemma~\ref{lem:pivot} says that such a pivot is a genuine cut point for the FPP metric: every geodesic from the left of the pivot to the right of the pivot must pass through it.

\begin{figure}[htbp]
\centering
\begin{tikzpicture}[
    x=.75cm,y=.75cm,>=Latex,
    every node/.style={font=\small},
    pt/.style={circle,fill,inner sep=1.4pt},
    lightedge/.style={thick},
    heavyedge/.style={thick,dashed}
]

    % baseline
    \draw (-6.3,0) -- (6.3,0);

    % main coordinates
    \coordinate (L)  at (-6,0);
    \coordinate (Ln) at (-4,0);
    \coordinate (Lm) at (-2,0);
    \coordinate (V)  at (0,0);
    \coordinate (Rp) at (2,0);
    \coordinate (Rq) at (4,0);
    \coordinate (R)  at (6,0);

    % visible vertices
    \node[pt,label=below:$v-3$] at (L) {};
    \node[pt,label=below:$v-2$] at (Ln) {};
    \node[pt,label=below:$v-1$] at (Lm) {};
    \node[pt,label=below:$v$]   at (V) {};
    \node[pt,label=below:$v+1$] at (Rp) {};
    \node[pt,label=below:$v+2$] at (Rq) {};
    \node[pt,label=below:$v+3$] at (R) {};

    % light incident edges from v
    \draw[lightedge] (V) to[out=155,in=25] (Lm);
    \draw[lightedge] (V) to[out=155,in=25] (Ln);
    %\draw[lightedge] (V) to[out=145,in=35] (L);
    \draw[lightedge] (V) to[out=25,in=155] (Rp);
    \draw[lightedge] (V) to[out=25,in=155] (Rq);
    %\draw[lightedge] (V) to[out=35,in=145] (R);

    % heavy crossing edges over v (representative examples)
    \draw[heavyedge] (Ln) to[out=40,in=140] (Rp);
    \draw[heavyedge] (Lm)  to[out=25,in=155] (Rp);
    \draw[heavyedge] (Lm)  to[out=25,in=155] (Rq);

    % legend on the right (avoids overlap)
    \draw[lightedge] (6.9,.25) -- (7.9,.25);
    \node[anchor=west,align=left] at (8.1,.25)
    {incident edges from $v$\\have weight $<a$};

    \draw[heavyedge] (6.9,1.25) -- (7.9,1.25);
    \node[anchor=west,align=left] at (8.1,1.25)
    {crossing edges over $v$\\have weight $>2a$};

\end{tikzpicture}
\caption{Local structure of a pivot node for $\ell=3$. The pivot condition requires $2(\ell-1)$ many  incident edges from $v$ to have weight $<a$, while every edge crossing over $v$ to have weight $>2a$. The figure shows only representative examples of such edges.}
\label{fig:pivot-local}
\end{figure}
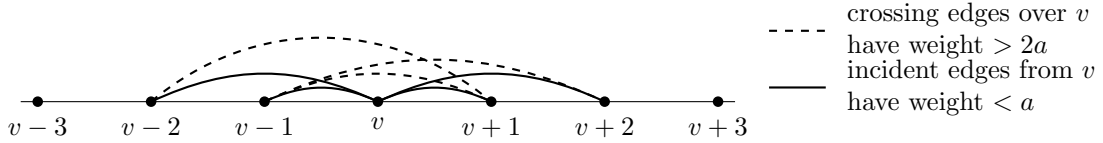

\begin{lem}\label{lem:pivot}
    Let $v$ be a pivot node according to definition~\ref{defn:B1-pivot}. For every node $ u,w$ with $u<v<w$, the geodesic from $u$ to $w$ passes through $v$.
\end{lem}

\begin{proof}[Proof of Lemma~\ref{lem:pivot}]
	Let $v$ be a pivot node. Suppose a geodesic from $u$ to $w$ for some $u<v<w$ does not include the pivot node $v$. Then there must be adjacent nodes $x<y$ in the path with $x<v<y, y-x\le \ell$, forcing the path to use the edge $(x,y)$ instead of $(x,v)$ and $(v,y)$. By the pivot node condition, we have
	$
		 \go_{x,v} +\go_{v,y} < 2a < \go_{x,y},
	$ which leads to a contradiction.
	Therefore, every geodesic must pass through the pivot node.
 \end{proof}

However, under the complementary Condition~B.\ref{ass:B2}, an ordinary one-vertex pivot cannot exist. We then use a finite chain of locally favorable vertices. After rescaling, we assume without loss of generality that
\begin{equation*}
    \inf\supp(\go)=1,
    \qquad
    \sup\supp(\go)=1+s, \text{ for some } s\in(0,1],
\end{equation*}
Thus, there exists an integer $k\ge 2$
such that
\begin{equation}\label{eq:B2-light-heavy-positive}
\pr(1\le \go<1+1/(k+1))>0
\text{ and }
\pr(\go>1+2/k)>0.	
\end{equation}
We call an edge \emph{light} if $1\le\go_e<1+1/(k+1)$ and \emph{heavy} if $\go_e>1+2/k$.
\begin{defn}[Generalized pivot]\label{defn:B2-kpivot}
A vertex $x$ is called a \emph{generalized pivot} if the set
\[
    S_x:=\{x,x+\ell,\ldots,x+(k-1)\ell\}.
\]
satisfies the following two conditions
\begin{enumeratei}
    \item every edge with one endpoint in $S_x$ is light, except no restriction on the two edges $(x-\ell,x)$ and $(x+(k-1)\ell,x+k\ell)$;
    \item every edge crossing over one of the vertices in $S_x$ is heavy.
\end{enumeratei}
\end{defn}

Note that the original pivot definition under Condition~B.\ref{ass:B1}, corresponds to taking $k=1$ in the generalized pivot case with $a=3/2$. The analogue of Lemma~\ref{lem:pivot} in this case is the following statement.
\begin{lem}\label{lem:B2-kpivot-forcing}
Let $x$ be a generalized pivot.  If $u<x$ and 
$
    w>x+(k-1)\ell,
$
then every geodesic from $u$ to $w$ passes through $x$.
\end{lem}

The proof is given in Section~\ref{sec:aux}. The idea is that a path avoiding the chain $S_x$ must use at least $k$ heavy crossing edges, whereas the path through the chain uses only $(k+1)$ light edges; the choice $\go<1+1/(k+1)$ makes the path strictly cheaper. 

In either case, ordinary pivots under Condition~B.\ref{ass:B1} or generalized pivots under Condition~B.\ref{ass:B2}, we obtain the pivot nodes in increasing order on the infinite $\ell$-spread-out line, and we denote them by
\[
    0\le \rho_1<\rho_2<\rho_3<\cdots .
\]
For consecutive pivot nodes, we define
\[
    \sfP_i:=\rho_{i+1}-\rho_i,
    \qquad
    \sfT_i:=T(\rho_i,\rho_{i+1}),
    \qquad
    \sfH_i:=H(\rho_i,\rho_{i+1}) .
\]
Let
\begin{align}\label{eq:kappa-def}
    \gk_n:=
    \begin{cases}
        \max\{i\ge1:\rho_i\le n\},
        &\text{under Condition~B.\ref{ass:B1}},\\[1mm]
        \max\{i\ge1:\rho_i+(k-1)\ell\le n\},
        &\text{under Condition~B.\ref{ass:B2}},
    \end{cases}
\end{align}
with the convention that $\max\eset=0$. For the renewal-time formulas below, we use the convention $k=1$ under Condition~B.\ref{ass:B1} and set
\begin{align}\label{eq:Mk-def}
    M_k:=(k-1)\ell.
\end{align}
Thus, $M_k=0$ in the ordinary-pivot case. With this convention, 
\[
\gk_n=\max\{i\ge1:\rho_i+M_k\le n\}
\] 
in both cases. The pivot forcing property implies that the passage time and the hop count admit the regenerative decompositions
\begin{align}
    T(0,n)
    &=
    \sum_{i=1}^{\gk_n-1}\sfT_i+\sfRT_n,
    \label{eq:T-decomp}\\
    H(0,n)
    &=
    \sum_{i=1}^{\gk_n-1}\sfH_i+\sfRH_n .
    \label{eq:H-decomp}
\end{align}

where $\sfRT_n$ and $\sfRH_n$ denote the boundary residual terms for the passage time and hop count, respectively. These residuals are explicitly given by
\begin{align}\label{eq:residual-def}
    \sfRT_n=T(0,\rho_1)+T(\rho_{\gk_n},n),
    \text{ and }
    \sfRH_n=H(0,\rho_1)+H(\rho_{\gk_n},n)
    \quad\text{for }\gk_n\ge 1.
\end{align}
If $\gk_n = 0$, we have $\sfRT_n = T(0,n)$ and $\sfRH_n = H(0,n)$.
By Lemma~\ref{lem:B1-pivot-env} in the ordinary-pivot case and Lemma~\ref{lem:B2-pivot-env} in the generalized-pivot case, the sequence
$
    \{(\sfP_i,\sfT_i,\sfH_i)\}_{i\ge1}
$
is i.i.d.~This is the renewal structure used throughout the microscopic analysis.

\subsection{Existing Results}\label{ssec:lit}

While our study is partially motivated by the spatial and complex-network phenomena discussed earlier in Section~\ref{sec:intro}, it is crucial to distinguish our deterministic spread-out geometry from the most common probabilistic long-range constructions. In small-world and long-range percolation models, long edges are typically added randomly, often with probabilities depending on distance~\cite{wattsstrogatz98,newmanwatts99,kleinberg00,biskup04}. In densification models, the average degree or link density may itself grow with the system size~\cite{leskovec07}. By contrast, in the microscopic regime studied here, $\ell$ is fixed, and all edges of range at most $\ell$ are present deterministically. The randomness lies exclusively in the passage times. Consequently, the relevant mathematical comparison for our work is not mean-field FPP or random long-range percolation, but rather finite-range, one-dimensional FPP with a regenerative structure.

The present paper is concerned with a one-dimensional finite-range version of FPP, and its method falls within the regeneration approach to one-dimensional first- and last-passage models.
The closest predecessor is Ahlberg~\cite{a15}, who studies FPP on one-dimensional periodic graphs and proves laws of large numbers, central limit theorems, laws of the iterated logarithm, and Donsker theorems for passage times and geodesic lengths by exposing a regenerative structure. The fixed-$\ell$ spread-out line belongs to the same broad one-dimensional class, and the Gaussian part of our result is consistent with this regenerative picture. Our contribution is different in that we exploit the explicit cut geometry of the $\ell$-spread-out line to obtain sharp moment thresholds, a precise one-block tail asymptotic, and a stable fluctuation regime.

There are also explicit Markovian approaches to low-dimensional FPP.  Renlund~\cite{renlund10} studies FPP with exponential times on a ladder through an associated Markov chain, and Schlemm~\cite{schlemm11} analyzes FPP on the ladder via finite-dimensional Markov transition kernels and obtains a central limit theorem.  These works are methodologically different from ours: they rely on a tractable finite-dimensional state description, whereas our pivot construction works directly at the level of the FPP geometry and applies to general non-degenerate edge-weight distributions in the fixed-range regime.

The word ``spread-out'' appears in the lace-expansion literature, where high-dimensional models are modified by allowing long but finite-range bonds to access mean-field behavior; see, for instance, \cite{slade06lace,hvdhs03lace}.  Our use of the spread-out geometry is different.  All edges of range at most $\ell$ are present, the randomness lies in the passage times, and the main objects are geodesics, hop counts, and passage-time fluctuations rather than critical two-point functions or cluster sizes.  Thus, the lace expansion is not used in our proof; in the fixed-range regime, the relevant structure is regenerative.

We do not claim that the concept of pivot nodes is new.  Rather, the novelty lies in applying this structure to the specific geometry of undirected finite-range FPP and in extracting the exact heavy-tail consequences of that geometry.  The spread-out line has
$
    \lc={\ell(\ell+1)}/2
$
many edge-disjoint routes of similar length through a unit cut.  Hence, a regeneration-block passage time can have a finite $k$-th moment even when the edge weight itself does not.  We identify the exact threshold $k/{\lc}$ for the $k$-th moment of the block passage time, and prove the precise one-block tail asymptotic
\begin{align*}
    \pr(\sfT_1>x)
    =
    \mu_{\sfP}\,\bar F(x)^{\lc}\cdot (1+o(1)) \text{ as } x\to\infty
\end{align*}
where $\mu_{\sfP}$ is the expected gap between two pivot nodes.
This tail asymptotic yields a distributional phase transition: above the second-moment threshold, the regeneration rewards produce Gaussian and Brownian fluctuations, while below it, they produce a totally right-skewed stable law.  At the process level, the same tail asymptotic gives a stable L\'evy limit for the renewal-skeleton process.  This last statement should be read in this skeleton sense: the corresponding deterministic-target process $\{T(0,\lfloor nt\rfloor):t\ge0\}$ requires additional control of endpoint and intra-block residuals and is treated below as a separate open problem.  In particular, the stable phase is not a direct consequence of earlier finite-variance regenerative FPP theory; it relies on the exact regular variation of the block reward, which is a geometric consequence of the finite cut constant $\lc$.

Regeneration ideas are also central in directed last-passage models on the line.  Foss, Martin, and Schmidt~\cite{fms14} construct pivot nodes for a long-range directed LPP model and obtain a strong law and a functional central limit theorem in the finite-second-moment case.  In their infinite-variance regime, however, the limiting object is a continuous last-passage percolation model on $[0,1]$, rather than the stable L\'evy skeleton process arising from our one-block tail asymptotic.  This contrast highlights the role of the first-passage metric and the undirected finite-range cut geometry in the present paper.

A broader directed-graph literature uses the same structural idea under different names: skeleton points, posts, renewal points, $c$-renewal points, and renovating events.  The classical directed acyclic model goes back to Barak and Erd\H{o}s~\cite{barakerdos84}.  Foss and Konstantopoulos~\cite{fk03} developed extended renovation theory and limit theorems for stochastic ordered graphs.  Denisov, Foss, and Konstantopoulos~\cite{dfk12} used regenerative skeletons to prove limit theorems for directed graphs on the line and on finite slabs.  Foss and Zachary~\cite{fz13} gave a general framework for regenerative structures whose defining random times may depend simultaneously on the past and on the future.  This point is relevant here because a pivot condition is local but two-sided: deciding that a vertex is a pivot requires inspecting edge weights on both sides of the vertex.

Further developments of the skeleton-point structure include Foss and Konstantopoulos~\cite{fk18}, Konstantopoulos, Logachov, Mogulskii, and Foss~\cite{klmf21}, Foss, Konstantopoulos, and Pyatkin~\cite{fkp23}, and the survey of Foss, Konstantopoulos, Mallein, and Ramassamy~\cite{fkmr24}.  The one-dimensional ordered setting is typically regenerative and therefore, under finite second moments, leads to Gaussian fluctuations. By contrast, two-dimensional directed models may exhibit Tracy--Widom fluctuations. For example, Konstantopoulos and Trinajsti\'c~\cite{kt13} prove a Tracy--Widom limit for longest paths in a directed random graph on $\dZ^2$ in a scaling regime where the two side lengths grow at different powers.  This should not be interpreted as a generic ``high-dimensional'' phenomenon.  Rather, it reflects the additional two-dimensional directed path geometry, and is closer in spirit to the Tracy--Widom limits in exactly solvable $(1+1)$--dimensional last-passage growth models such as Johansson's model~\cite{johansson00}.

The pivot method should also be contrasted with deterministic block decompositions.  When $\ell$ is fixed, pivots occur with positive probability and force all left-to-right geodesics through random cut points.  This gives an exact regenerative decomposition of the original FPP metric.  Conversely, as the range or graph width grows with $n$, the probability of finding such local pivots typically degenerates.  In such regimes, deterministic blocking, approximation errors, and triangular-array central limit theorems become natural.  Chatterjee and Dey~\cite{cd13} use this type of approach for FPP across growing thin cylinders.  Thus, the random pivot decomposition used here and deterministic block decompositions are complementary tools for different geometric regimes.

Finally, the martingale central limit theorem perspective is compatible with our proof, but is not the main structural input.  Classical martingale CLTs~\cite{brown71,hallheyde80,mcleish74} can be used after regeneration, but the pivot construction is the input that makes the increments explicit and independent.

The probabilistic limit theorems used after the regeneration structure is established are classical.  The Gaussian and Brownian limits follow from renewal-reward theory and invariance principles, while the stable and L\'evy skeleton limits follow from regular variation and stable convergence for partial sums; see, for example, \cite{mitov2014,whitt02,resnick07,bks12}.

\subsection{Notations}\label{ssec:notation}

Throughout the paper, dependence on the fixed range $\ell$ and on the weight distribution is usually suppressed.
For random variables $X,Y$, we write $X\preceq Y$ for stochastic domination, and for $p\ge1$ we write
\[
    \|X\|_p:=\big(\E |X|^p\big)^{1/p}.
\]
The tail of the edge-weight distribution is denoted by
\[
    \bar F(x):=\pr(\go>x).
\]

Unless a finite graph is explicitly specified, all microscopic constructions are made on the infinite $\ell$-spread-out line.  We write
\[
    \cE^{(\ell)}
    :=
    \{u,v\in\dZ:u<v,\ v-u\le \ell\}.
\]
For integers $x<y$, let
\begin{align}\label{eq:interval-env-notation}
    \cE[x,y]
    &:=
    \{(u,v)\in\cE^{(\ell)}:x\le u<v\le y\},
    \\
    \Env[x,y]
    &:=
    \{\go_{u,v}:(u,v)\in\cE[x,y]\}.
\end{align}
If $I\subseteq\dZ$ is an interval and $u,v\in I$, then $T_I(u,v)$ denotes the passage time restricted to the induced subgraph on $I$. When no subscript is present, $T(u,v)$ is the unrestricted passage time on the infinite $\ell$-spread-out line.  A geodesic means a path attaining this passage time.  We denote by
$
    H(u,v)
$
the minimum number of edges among all geodesics from $u$ to $v$. We will also use 
\[H_n:=H(0,n)\text{ and } T_n:=T(0,n).
\]

If there are multiple minimum-hop geodesics, we uniquely select one, denoted by $\pi^\star(u,v)$, using a standard shift-invariant lexicographic order. This selection is consistent under path concatenation at pivot points. In particular, we write $\pi_n^\star := \pi^\star(0,n)$. This convention is needed only for the geometric statistics of geodesics and does not affect the definitions of $T(u,v)$ and $H(u,v)$.
Moreover, this selection is strictly compatible with path concatenation: whenever every geodesic from $u$ to $w$ must pass through a vertex $v$ (with $u<v<w$), we have
\[
\pi^\star(u,w) = \pi^\star(u,v)\circ\pi^\star(v,w),
\]
because any minimum-hop geodesic from $u$ to $w$ naturally splits into minimum-hop subpaths at $v$, and the lexicographic order preserves this spatial decomposition.

The pivot nodes are denoted by
\[
    0\le \rho_1<\rho_2<\cdots.
\]
Under Condition~B.\ref{ass:B1}, $\rho_i$ is an ordinary pivot node. Under Condition~B.\ref{ass:B2}, $\rho_i$ denotes the left endpoint of a generalized pivot chain. In the latter case, the chain associated with $\rho_i$ is $\{\rho_i,\rho_i+\ell,\ldots,\rho_i+(k-1)\ell\}$.  The renewal count $\gk_n$ is defined in~\eqref{eq:kappa-def}; in particular, under Condition~B.\ref{ass:B2}, only generalized pivots whose full chain lies to the left of $n$ are counted.

For renewal block variables, define the gap of consecutive pivots, passage-time, and hop-count by
\begin{equation}\label{eq:block-variable-notation}
    \sfP_i:=\rho_{i+1}-\rho_i,
    \qquad
    \sfT_i:=T(\rho_i,\rho_{i+1}),
    \qquad
    \sfH_i:=H(\rho_i,\rho_{i+1}).
\end{equation}
We also use the partial sums
\[
    S_m^{\sfP}:=\sum_{i=1}^m\sfP_i,
    \qquad
    S_m^{\sfT}:=\sum_{i=1}^m\sfT_i,
    \qquad
    S_m^{\sfH}:=\sum_{i=1}^m\sfH_i,
\]
and the renewal counting process
\[
    \sfN_t:=\max\{m\ge0:S_m^{\sfP}\le t\}.
\]
The endpoint residuals in the point-to-point decomposition are denoted by $\sfRT_n$ and $\sfRH_n$ and are defined in~\eqref{eq:residual-def}.

Whenever the following expectations are finite, set
\begin{align}\label{eq:mu-nu-eta}
    \mu_{\sfP}:=\E\sfP_1,
    \qquad
    \mu_{\sfT}:=\E\sfT_1,
    \qquad
    \mu_{\sfH}:=\E\sfH_1 .
\end{align}
For the passage-time and joint CLTs, write
\begin{align}\label{eq:sigma12}
    \gs_{\sfP}^2&:=\var(\sfP_1),
    \qquad
    \gs_{\sfT}^2:=\var(\sfT_1),
    \qquad
    \gs_{\sfH}^2:=\var(\sfH_1),
    \notag\\
    \gs_{\sfP\sfT}&:=\cov(\sfP_1,\sfT_1),
    \qquad
    \gs_{\sfP\sfH}:=\cov(\sfP_1,\sfH_1),
    \qquad
    \gs_{\sfT\sfH}:=\cov(\sfT_1,\sfH_1).
\end{align}
The asymptotic passage-time variance is
\begin{align}\label{eq:var-CLT-def}
\gs_{\mathrm{FPP}}^2
:=
\frac1{\mu_{\sfP}}
\var\left(\sfT_1-\frac{\mu_{\sfT}}{\mu_{\sfP}}\sfP_1\right) 
=
\frac{\mu_{\sfP}^2\gs_{\sfT}^2+
      \mu_{\sfT}^2\gs_{\sfP}^2-
      2\mu_{\sfP}\mu_{\sfT}\gs_{\sfP\sfT}}
     {\mu_{\sfP}^3}.
\end{align}
For the hop count, define
\begin{align}\label{eq:hop-var-def}
    \gs_{\mathrm{Hop}}^2
    :=
    \frac1{\mu_{\sfP}}
    \var\left(\sfH_1-\frac{\mu_{\sfH}}{\mu_{\sfP}}\sfP_1\right)
    =
    \frac{\mu_{\sfP}^2\gs_{\sfH}^2+
          \mu_{\sfH}^2\gs_{\sfP}^2-
          2\mu_{\sfP}\mu_{\sfH}\gs_{\sfP\sfH}}
         {\mu_{\sfP}^3}.
\end{align}
Finally, the joint covariance coefficient between the passage-time and hop-count limits is
\begin{align}\label{eq:TH-cov-def}
    \gs_{\mathrm{TH}}
    :=
    \frac1{\mu_{\sfP}}
    \cov\left(
        \sfT_1-\frac{\mu_{\sfT}}{\mu_{\sfP}}\sfP_1,\,
        \sfH_1-\frac{\mu_{\sfH}}{\mu_{\sfP}}\sfP_1
    \right)
    =
    \frac{\mu_{\sfP}^2\gs_{\sfT\sfH}
          -\mu_{\sfP}\mu_{\sfT}\gs_{\sfP\sfH}
          -\mu_{\sfP}\mu_{\sfH}\gs_{\sfP\sfT}
          +\mu_{\sfT}\mu_{\sfH}\gs_{\sfP}^2}
         {\mu_{\sfP}^3},
\end{align}
and the two-dimensional covariance matrix is
\begin{align}\label{eq:joint-cov-def}
\Sigma_\ell
:=
\frac1{\mu_{\sfP}}
\cov\left(
\begin{pmatrix}
    \sfT_1-\frac{\mu_{\sfT}}{\mu_{\sfP}}\sfP_1\\[1mm]
    \sfH_1-\frac{\mu_{\sfH}}{\mu_{\sfP}}\sfP_1
\end{pmatrix}
\right).
\end{align}

\begin{rem}
Although the auxiliary renewal construction depends on the choice of admissible pivot parameters---namely $a$ under Condition~B.\ref{ass:B1} and $k$ under Condition~B.\ref{ass:B2}---the limiting constants appearing in the laws of large numbers and central limit theorems are intrinsic to the model and therefore independent of these parameters. Consequently, we suppress this auxiliary dependence in our notation.
\end{rem}

\subsection{Structure of the article}

The paper is organized as follows.
Section~\ref{sec:pivot} develops the pivot-regeneration framework under both Condition~B.\ref{ass:B1} and Condition~B.\ref{ass:B2}. There, we prove the renewal decomposition, the block moment estimates, and the endpoint residual bounds, and then derive Theorem~\ref{thm:micro} (i)--(ii) and Theorem~\ref{thm:micro-hop-clt}.  
Section~\ref{sec:stable} treats the heavy-tailed stable regime under Condition~A.\ref{ass:A2}; the key input is the one-block tail asymptotic with exponent $\lc=\ell(\ell+1)/2$, which yields Theorem~\ref{thm:micro-stable}. 
Section~\ref{sec:aux} contains the deferred structural proofs for the ordinary and generalized pivot constructions.

%%%%%%%%%%%%%%%%%%%%%%%%%%%%%%%%%%%%%%%%%%%%%%%%%%%%%%%%%%%%%%%%%%%%%%%%%%%%%%%%%%%
\section{Pivot Control and Gaussian CLT}\label{sec:pivot}
%%%%%%%%%%%%%%%%%%%%%%%%%%%%%%%%%%%%%%%%%%%%%%%%%%%%%%%%%%%%%%%%%%%%%%%%%%%%%%%%%%%
In this section, we prove Theorem~\ref{thm:micro}--\eqref{thm:mean},~\eqref{thm:clt}, and Theorem~\ref{thm:micro-hop-clt} by establishing an exact i.i.d.~renewal structure for the first-passage time and hop count. While the central limit theorem follows readily under the finite second moment assumption (Condition~A.\ref{ass:A1}), the heavy-tailed regime (Condition~A.\ref{ass:A2}) requires a precise identification of the moment threshold $k/\lc$ for the block passage times. Once this moment control is established, parts~\eqref{thm:mean} and~\eqref{thm:clt} of Theorem~\ref{thm:micro} follow naturally from standard results in renewal theory~\cite{mitov2014}.
We begin by formulating a precise version of our main results, explicitly identifying the centering and scaling constants that were left unspecified in Theorems~\ref{thm:micro} and~\ref{thm:micro-hop-clt}.
 
\begin{thm}[First-passage time and hop-count limits]
\label{thm:LLN-CLT}
Assume that $\ell\ge2$ and $\go$ is not deterministic. 
\begin{enumeratea}
\item If either $\E\go<\infty$, or Assumption~A.\ref{ass:A2} holds with
$\gc>1/\lc$, then
\begin{align}\label{eq:B1-weak-lln}
    \frac{T_n}{n}
    \stackrel{\pr}{\to}
    \frac{\mu_{\sfT}}{\mu_{\sfP}}.
\end{align}

\item If either $\E\go^2<\infty$, or Assumption~A.\ref{ass:A2} holds with $\gc>2/\lc$ then
\begin{align}\label{eq:B1-clt}
    n^{-1/2}\left(T_n-n\cdot {\mu_{\sfT}}/{\mu_{\sfP}}\right)
    \Longrightarrow
    \N(0,\gs_{\mathrm{FPP}}^2),
\end{align}
where $\gs_{\mathrm{FPP}}^2$ is defined in~\eqref{eq:var-CLT-def}.

\item For the hop-count, we have
\begin{align}\label{eq:hop-lln}
    \frac{H_n}{n} \to \frac{\mu_{\sfH}}{\mu_{\sfP}} \quad \text{a.s.}, 
    \quad \text{and} \quad 
    \frac{\E H_n}{n} \to \frac{\mu_{\sfH}}{\mu_{\sfP}}.
\end{align}
Moreover, the central limit theorem holds:
\begin{align}\label{eq:hop-clt}
    n^{-1/2}\left(H_n-n\cdot {\mu_{\sfH}}/{\mu_{\sfP}}\right) \Longrightarrow \N(0,\gs_{\mathrm{Hop}}^2),
\end{align}
where $\gs_{\mathrm{Hop}}^2$ is defined in~\eqref{eq:hop-var-def}.

\item If either $\E\go^2<\infty$, or Assumption~A.\ref{ass:A2} holds with
$\gc>2/\lc$, then the joint central limit theorem holds:
\begin{align}\label{eq:B1-joint-clt}
   n^{-1/2}
    \begin{pmatrix}
        T_n-n\cdot {\mu_{\sfT}}/{\mu_{\sfP}}\\[2mm]
        H_n-n\cdot {\mu_{\sfH}}/{\mu_{\sfP}}
    \end{pmatrix}
    \Longrightarrow
    \N_2(\mvzero,\Sigma_\ell),
\end{align}
where $\Sigma_\ell$ is defined in~\eqref{eq:joint-cov-def}. Since
$
    \E H_n
    =
    n\cdot {\mu_{\sfH}}/{\mu_{\sfP}}+O(1),
$
the second coordinate in~\eqref{eq:B1-joint-clt} may equivalently be
replaced by $(H_n-\E H_n)/\sqrt n$.\end{enumeratea}
\end{thm}

\begin{rem}[Centering constants and moment conditions]
The joint convergence statement in Theorem~\ref{thm:LLN-CLT} depends subtly on whether Condition~\textup{B.\ref{ass:B1}} or Condition~\textup{B.\ref{ass:B2}} is assumed. Under Condition~\textup{B.\ref{ass:B1}}, the exact expectation $\E T_n$ might not exist, so we center $T_n$ by the deterministic linear term $n(\mu_{\sfT}/\mu_{\sfP})$ instead of $\E T_n$. On the other hand, under Condition~\textup{B.\ref{ass:B2}}, the endpoint residual passage time is uniformly bounded in $L^q$ for any $q \ge 1$ by Lemma~\ref{lem:end-residual}. This guarantees almost sure convergence for the law of large numbers and allows us to equivalently state the central limit theorem using the exact centering $T_n - \E T_n$.
\end{rem}

The proof of Theorem~\ref{thm:LLN-CLT} proceeds in three straightforward steps. First, we establish the foundations for the pivot decomposition and the resulting i.i.d.~renewal structure (Lemmas~\ref{lem:B1-pivot-env} and~\ref{lem:B2-pivot-env}, with proofs deferred to Section~\ref{sec:aux}. Next, Proposition~\ref{prop:k-moment} identifies the exact integrability thresholds for the block rewards $\sfT_1$. Finally, Lemma~\ref{lem:end-residual} controls the boundary residual terms ($\sfRT_n$ and $\sfRH_n$), which directly allows us to apply the standard renewal-reward limits (Lemma~\ref{lem:renewal-reward-clt}) to conclude the proof.

To initialize this regenerative structure, we first verify that the pivot nodes occur frequently enough. Under Condition~B.\ref{ass:B1}, we define the probability of a node being an ordinary pivot as
\begin{align}\label{eq:B1-pivot-prob}
    p_{a,\ell} := \pr(0 \text{ is a pivot node}).
\end{align}
The pivot condition~\eqref{def:pivot} requires $2(\ell-1)$ incident edges adjacent to $0$ to be lighter than $a$, and the $\ell(\ell-1)/2$ edges crossing over vertex $0$ to be heavier than $2a$. By the independence of the edge weights and Condition~B.\ref{ass:B1}, we have 
\begin{align*}
    p_{a,\ell} = \pr(\go < a)^{2(\ell-1)} \pr(\go > 2a)^{\ell(\ell-1)/2} > 0.
\end{align*}
Therefore, ordinary pivot nodes occur infinitely often along the line. This allows us to partition the graph into independent and identically distributed blocks, as formalized in the following lemma.

\begin{lem}[Pivot regeneration blocks]\label{lem:B1-pivot-env}
Fix $\ell\ge2$ and let $a>0$ satisfy
Condition~\textup{B.\ref{ass:B1}}.  Let
$
    0<\rho_1<\rho_2<\rho_3<\cdots
$
be the ordinary pivot nodes, and let $\gk_n$ be as in
\eqref{eq:kappa-def}.  We use the interval edge set $\cE[x,y]$ and
the interval environment $\Env[x,y]$ defined in
\eqref{eq:interval-env-notation}, and the block notation
\eqref{eq:block-variable-notation}.  Then the following holds.

\begin{enumeratei}
\item \label{piv-i}\textbf{Pivot renewal-block localization.}
If $\rho_i<\rho_j$ are two ordinary pivot nodes, then there exists a minimum-hop geodesic from $\rho_i$ to $\rho_j$ using only edges in
$\cE[\rho_i,\rho_j]$.  Consequently,
$
    T(\rho_i,\rho_j)
    \text{ and }
    H(\rho_i,\rho_j)
$
are measurable with respect to $\Env[\rho_i,\rho_j]$.

\item \label{piv-ii}\textbf{Renewal-block i.i.d.~structure.}
The block sequence
\begin{align}\label{eq:pivot-iid-str}
    \left\{
        \left(
            \sfP_i,\,
            \Env[\rho_i,\rho_{i+1}]
        \right)
    \right\}_{i\ge1}
\end{align}
is i.i.d.~In particular,
$
    \{(\sfP_i,\sfT_i,\sfH_i)\}_{i\ge1}
$
is i.i.d.

\item \label{piv-iii}\textbf{Geometric domination of the gaps.}
Ordinary pivots occur infinitely often almost surely, and if
$
    \fs_\ell:=2\ell-1,
$
then
\begin{align}\label{eq:pivot-gap-domin}
    \sfP_i
    \preceq
    \fs_\ell\,\mathrm{Geom}(p_{a,\ell}),
    \quad i\ge1.
\end{align}
In particular, $\sfP_1$ has finite moments of all orders.  Consequently,
\begin{align}\label{eq:pivot-renewal-slln}
    \frac{\gk_n}{n}
    \to
    \frac{1}{\E\sfP_1}
    \quad\text{a.s.}
\end{align}
\end{enumeratei}
\end{lem}

Analogously, under Condition~B.\ref{ass:B2}, we set the probability of a generalized pivot as
\begin{align*}
    p_{\mathrm{gen},\ell}
    :=
    \pr(0 \text{ is a generalized pivot node}).
\end{align*}
As in the derivation of~\eqref{eq:B1-pivot-prob}, the definition of a generalized pivot involves only finitely many independent edges satisfying specified light and heavy inequalities, each of which occurs with positive probability by~\eqref{eq:B2-light-heavy-positive}. Since an edge incident to a vertex in the pivot chain cannot strictly cross another vertex in the same chain, no edge is required to be both light and heavy.  
Hence $p_{\mathrm{gen},\ell}>0$, and we obtain the analogous renewal structure below.

\begin{lem}[Generalized pivot regeneration blocks]\label{lem:B2-pivot-env}
Let
$
    0<\rho_1<\rho_2<\rho_3<\cdots
$
be the generalized pivots, and let $\gk_n$ be as in~\eqref{eq:kappa-def}.  We use the interval edge set $\cE[x,y]$ and the interval environment $\Env[x,y]$ defined in~\eqref{eq:interval-env-notation}, and the block notation~\eqref{eq:block-variable-notation}. Then the following holds.

\begin{enumeratei}
\item\label{gen-piv-i} \textbf{Generalized pivot renewal-block localization.}
If $\rho_i<\rho_j$ are two generalized pivots, then there
exists a minimum-hop geodesic from $\rho_i$ to $\rho_j$ using only
edges in $\cE[\rho_i,\rho_j]$.  Consequently,
$T(\rho_i,\rho_j)\text{ and }H(\rho_i,\rho_j)
$ are measurable with respect to $\Env[\rho_i,\rho_j]$.

\item\label{gen-piv-ii} \textbf{Renewal-block i.i.d.~structure.}
The block sequence
\begin{align}\label{eq:B2-env-iid}
    \left\{
        \left(
            \sfP_i,\,
            \Env[\rho_i,\rho_{i+1}]
        \right)
    \right\}_{i\ge1}
\end{align}
is i.i.d.~In particular,
$
    \{(\sfP_i,\sfT_i,\sfH_i)\}_{i\ge1}
$
is i.i.d.

\item\label{gen-piv-iii} \textbf{Geometric domination of the gaps.}
Generalized pivots occur infinitely often almost surely, and if
$
    \fs_{\ell,k}:=(k+1)\ell-1,
$
then
\begin{align}\label{eq:B2-gap-geom-kpivot}
    \sfP_i
    \preceq
    \fs_{\ell,k}\,\mathrm{Geom}(p_{\mathrm{gen},\ell}),
    \quad i\ge1.
\end{align}
In particular, $\sfP_1$ has finite moments of all orders.  Consequently,
\begin{align}\label{eq:B2-pivot-renewal-slln}
    \frac{\gk_n}{n}
    \to
    \frac{1}{\E\sfP_1}
    \quad\text{a.s.}
\end{align}

\item\label{gen-piv-iv} \textbf{Deterministic bounds for bounded weights.}
All moments of $\sfT_i$ and $\sfH_i$ are finite, and
\begin{align}\label{eq:B2-block-bounds-kpivot}
    \left\lceil\frac{\sfP_i}{\ell}\right\rceil
    \le
    \sfH_i
    \le
    \sfT_i
    \le
    2\left\lceil\frac{\sfP_i}{\ell}\right\rceil,
    \qquad
    \sfH_i\le \sfP_i,
    \quad i\ge1.
\end{align}
\end{enumeratei}
\end{lem}

While the geometric constraints and probabilistic arguments underlying Lemmas~\ref{lem:B1-pivot-env} and~\ref{lem:B2-pivot-env} are intuitively straightforward, verifying the strict independence of the blocks requires carefully tracking the explored edge environments. To avoid interrupting the main analysis, we defer the rigorous proofs of these two structural lemmas to Section~\ref{sec:aux}.

Having established the exact i.i.d.~renewal structure for the sequence $\{(\sfP_i, \sfT_i, \sfH_i)\}_{i\ge1}$, we next address the integrability of these block variables. While the gap length $\sfP_1$ and the hop count $\sfH_1$ have finite moments of all orders due to their geometric domination by Lemmas~\ref{lem:B1-pivot-env} and~\ref{lem:B2-pivot-env}, the block passage time $\sfT_1$ is governed by the tail behavior of the edge weights. We now determine the necessary and sufficient condition for the existence of the $k$-th moment of $\sfT_1$.

\begin{prop}\label{prop:k-moment}
Let $k\in\dN$ and assume Condition~B.\ref{ass:B1}. Recall from~\eqref{eq:block-variable-notation} that $\sfT_1:=T(\rho_1,\rho_2)$.
\begin{enumeratei}
\item If Assumption~A.\ref{ass:A2} holds with $0<\gc<k/\lc$, then
$
\E \sfT_1^k=\infty.
$
\item If either Assumption~A.\ref{ass:A2} holds with $\gc>k/\lc$, or $\E\go^k<\infty$, then
$
\E \sfT_1^k<\infty.
$
\end{enumeratei}
\end{prop}

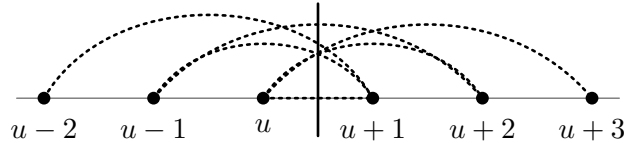
\begin{figure}[htbp]
\centering
\begin{tikzpicture}[
    x=1.45cm,y=1cm,>=Latex,
    line cap=round,line join=round,
    cutedge/.style={black, dotted, line width=1pt},
    cutline/.style={black, line width=1pt},
    vtx/.style={circle,fill=black,inner sep=1.7pt}
]
    % Illustration for \ell=3:
    % vertices u-2, u-1, u, u+1, u+2, u+3
    \coordinate (a0) at (0,0);
    \coordinate (a1) at (1,0);
    \coordinate (a2) at (2,0);
    \coordinate (a3) at (3,0);
    \coordinate (a4) at (4,0);
    \coordinate (a5) at (5,0);
    % baseline
    \draw[black!70] (-0.25,0) -- (5.25,0);
    % blue edges in C_u for \ell=3
    \draw[cutedge] (a2) -- (a3); % (u, u+1)
    % upper arcs: use out/in angles for a more natural curvature
    \draw[cutedge] (a1) to[out=58,in=122] (a3); % (u-1, u+1)
    \draw[cutedge] (a0) to[out=60,in=120] (a3); % (u-2, u+1)
    \draw[cutedge] (a1) to[out=50,in=130] (a4); % (u-1, u+2)

    % lower arcs
    \draw[cutedge] (a2) to[out=58,in=122] (a4); % (u, u+2)
    \draw[cutedge] (a2) to[out=50,in=130] (a5); % (u, u+3)

    % red vertical cut at u+1/2
    \draw[cutline] (2.5,-.50) -- (2.5,1.25);

    % vertices
    \foreach \p in {a0,a1,a2,a3,a4,a5}{
        \node[vtx] at (\p) {};
    }

    % labels
    \node[below=4pt] at (a0) {$u-2$};
    \node[below=4pt] at (a1) {$u-1$};
    \node[below=4pt] at (a2) {$u$};
    \node[below=4pt] at (a3) {$u+1$};
    \node[below=4pt] at (a4) {$u+2$};
    \node[below=4pt] at (a5) {$u+3$};
\end{tikzpicture}
\caption{Illustration of the unit cut $\cC_u$ in the case $\ell=3$ with $\abs{\cC_u}=\lc=6$. The dotted edges are precisely the edges in $\cC_u$, namely the edges crossing the boundary between $u$ and $u+1$. The vertical segment indicates the cut between the vertices $u,u+1$. Hence $X_u=\min_{e\in\cC_u}\go_e$ is the minimum weight among these dotted edges.}
\label{fig:unit-cut-Cu}
\end{figure}

To prove Proposition~\ref{prop:k-moment}, we analyze the combinatorial path structure of the $\ell$-spread-out line graph. A fundamental structural feature is the \emph{unit cut} separating the left and right sides of any vertex $u$, defined as the set of all edges crossing the boundary between $u$ and $u+1$:
\begin{align}\label{def:unit_cut}
    \cC_u := \{(x,y) \in \cE^{(\ell)} : x \le u < y\}.
\end{align}

Recall from~\eqref{def:ell_c} that $\lc=\ell(\ell+1)/2$. As illustrated in Figure~\ref{fig:unit-cut-Cu}, exactly $|\cC_u| = \sum_{d=1}^{\ell} d = \lc$ edges cross any such unit cut. 
This implies that any family of edge-disjoint paths traversing a renewal block can contain at most $\lc$ paths. As we will demonstrate, this upper bound is strictly achievable. This intrinsic network bottleneck is the fundamental reason why the integrability threshold for the block passage time $\sfT_1$ is precisely $k/\lc$. The following lemma formalizes this path structure, and a schematic construction is given in Figure~\ref{fig:path_structure}.

\begin{lem}\label{lem:path_structure}
Let $\rho_1<\rho_2$ be consecutive pivot nodes and set $j:=\rho_2-\rho_1$.
The following statements hold.
\begin{enumeratei}
\item If $j\le 2\ell-2$, then $T(\rho_1,\rho_2)<2a$, where $a$ is the fixed constant satisfying Condition~\textnormal{B.\ref{ass:B1}}.
\item If $j\ge 2\ell-1$, then there exist $\lc=\ell(\ell+1)/2$ paths from $0$ to $j$ such that, after deleting their first and last edges, the remaining edge sets are pairwise disjoint. Moreover, each such truncated path uses at most $j$ edges.
\end{enumeratei}
\end{lem}

\begin{proof}
By translation invariance, it suffices to consider paths from $0$ to $j$.
First, suppose $j\le 2\ell-2$. If $j=1$, since $0$ is a pivot node, $\go_{0,1}<a$.
Hence $T(0,1)\le \go_{0,1}<a<2a$.
If $2\le j\le 2\ell-2$, then choose
\begin{align*}
x\in[\max\{1,j-\ell+1\},\min\{\ell-1,j-1\}]\neq \eset.
\end{align*}
Then $1\le x\le \ell-1$ and $j-\ell+1\le x\le j-1$. By the pivot conditions at $0$ and $j$,
\begin{align*}
T(0,j)\le \go_{0,x}+\go_{x,j}<2a.
\end{align*}
Now, assume that $j>2\ell-2$. Let
\begin{align*}
\fL_0:=\{0,1,\ldots,\ell-1\},
\quad
\fR_j:=\{j-\ell+1,\ldots,j\}.
\end{align*}
We first show that there are at most $\lc$ many edge-disjoint paths from $\fL_0$ to $\fR_j$.
For any $u$ with $\ell-1\le u\le j-\ell$, any path from $\fL_0$ to $\fR_j$ must cross the boundary between $u$ and $(u+1)$. The set of edges crossing this boundary is exactly the unit cut $\cC_u$ defined in~\eqref{def:unit_cut}. Since $|\cC_u| = \lc$, any family of edge-disjoint paths from $\fL_0$ to $\fR_j$ has size at most $\lc$.

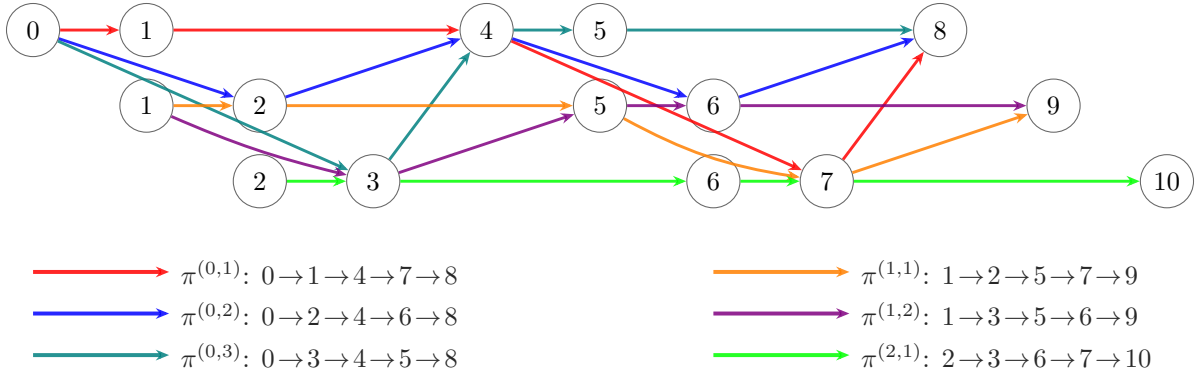
\begin{figure}[htbp]
\centering
\begin{tikzpicture}[
    % global style
    node/.style  = {circle, draw=black!60, fill=white,
                    inner sep=0pt, minimum size=20pt, font=\small},
    lbl/.style   = {font=\footnotesize\itshape, text=black!55},
    arr/.style   = {-{Stealth[length=5pt,width=4pt]},
                    line width=1.1pt, opacity=0.85},
    every path/.style = {bend angle=12},
    % column x-spacing and row y-spacing
    x=1.5cm, y=1cm
]
% ---------------------------------------------------------
% Nodes  (column, row) -> label
%   row 2 = top, row 1 = middle, row 0 = bottom
%   node 0 sits at (col=0, row=1)
% ---------------------------------------------------------
% col 1  (k = 1,2,3  ->  rows 2,1,0)
\node[node] (v1_1) at (1,2) {$0$};
\node[node] (v1_2) at (2,1) {$1$};
\node[node] (v1_3) at (3,0) {$2$};
% col 2  (k+b in {2,3,4}  ->  rows 2,1,0)
\node[node] (v2_2) at (2,2) {$1$};
\node[node] (v2_3) at (3,1) {$2$};
\node[node] (v2_4) at (4,0) {$3$};
% col 3  (k+ell+1 in {5,6,7}  ->  rows 2,1,0)
\node[node] (v3_5) at (5,2) {$4$};
\node[node] (v3_6) at (6,1) {$5$};
\node[node] (v3_7) at (7,0) {$6$};
% col 4  (2ell-b+3 in {6,7,8}  ->  rows 2,1,0)
\node[node] (v4_6) at (6,2) {$5$};
\node[node] (v4_7) at (7,1) {$6$};
\node[node] (v4_8) at (8,0) {$7$};
% col 5  (2(ell+1)+k in {9,10,11}  ->  rows 2,1,0)
\node[node] (v5_9)  at (9,2) {$8$};
\node[node] (v5_10) at (10,1) {$9$};
\node[node] (v5_11) at (11,0) {$10$};

% --------------------------------------------------------
% Paths
% --------------------------------------------------------
% pi^{(1,1)}: 0 -> 1 -> 2 -> 5 -> 8 -> 9   (red)
\draw[arr, red]
    (v1_1)  edge[bend left=0]   (v2_2)
    (v2_2)  edge[bend left=0]  (v3_5)
    (v3_5)  edge[bend right=0] (v4_8)
    (v4_8)  edge[bend left=0]  (v5_9);

% pi^{(1,2)}: 0 -> 1 -> 3 -> 5 -> 7 -> 9   (blue)
\draw[arr, blue]
    (v1_1)  edge[bend right=0]  (v2_3)
    (v2_3)  edge[bend left=0]   (v3_5)
    (v3_5)  edge[bend right=0]  (v4_7)
    (v4_7)  edge[bend left=0]  (v5_9);

% pi^{(1,3)}: 0 -> 1 -> 4 -> 5 -> 6 -> 9   (teal)
\draw[arr, teal]
    (v1_1)  edge[bend right=0] (v2_4)
    (v2_4)  edge[bend left=0]   (v3_5)
    (v3_5)  edge[bend right=0]  (v4_6)
    (v4_6)  edge[bend left=0]   (v5_9);

% pi^{(2,1)}: 0 -> 2 -> 3 -> 6 -> 8 -> 10  (orange)
\draw[arr, orange]
    (v1_2)  edge[bend left=0]   (v2_3)
    (v2_3)  edge[bend left=0]   (v3_6)
    (v3_6)  edge[bend right=10] (v4_8)
    (v4_8)  edge[bend left=0]   (v5_10);

% pi^{(2,2)}: 0 -> 2 -> 4 -> 6 -> 7 -> 10  (violet)
\draw[arr, violet]
    (v1_2)  edge[bend right=5]  (v2_4)
    (v2_4)  edge[bend left=0]   (v3_6)
    (v3_6)  edge[bend right=0]  (v4_7)
    (v4_7)  edge[bend left=0]   (v5_10);

% pi^{(3,1)}: 0 -> 3 -> 4 -> 7 -> 8 -> 11  (green)
\draw[arr, green]
    (v1_3)  edge[bend left=0]   (v2_4)
    (v2_4)  edge[bend right=0]  (v3_7)
    (v3_7)  edge[bend right=0]  (v4_8)
    (v4_8)  edge[bend right=0] (v5_11);

% --------------------------------------------------------
% Legend
% --------------------------------------------------------
\begin{scope}[shift={(0,-1.2)}]
  \def\lx{1}   \def\ly{0}   \def\ls{1.8cm}  % legend x start, y, col spacing

  % row 1
  \draw[arr, red,    line width=1.5pt] (\lx+0.0, \ly) -- +(\ls,0)
    node[right, font=\small, text=black] {$\pi^{(0,1)}$: $0\!\to\!1\!\to\!4\!\to\!7\!\to\!8$};

  \draw[arr, blue,   line width=1.5pt] (\lx+0.0, \ly-0.55) -- +(\ls,0)
    node[right, font=\small, text=black] {$\pi^{(0,2)}$: $0\!\to\!2\!\to\!4\!\to\!6\!\to\!8$};

  \draw[arr, teal,   line width=1.5pt] (\lx+0.0, \ly-1.10) -- +(\ls,0)
    node[right, font=\small, text=black] {$\pi^{(0,3)}$: $0\!\to\!3\!\to\!4\!\to\!5\!\to\!8$};

  \draw[arr, orange, line width=1.5pt] (\lx+6.0, \ly) -- +(\ls,0)
    node[right, font=\small, text=black] {$\pi^{(1,1)}$: $1\!\to\!2\!\to\!5\!\to\!7\!\to\!9$};

  \draw[arr, violet, line width=1.5pt] (\lx+6.0, \ly-.55) -- +(\ls,0)
    node[right, font=\small, text=black] {$\pi^{(1,2)}$: $1\!\to\!3\!\to\!5\!\to\!6\!\to\!9$};

  \draw[arr, green,  line width=1.5pt] (\lx+6.0, \ly-1.10) -- +(\ls,0)
    node[right, font=\small, text=black] {$\pi^{(2,1)}$: $2\!\to\!3\!\to\!6\!\to\!7\!\to\!10$};
\end{scope}
\end{tikzpicture}
\caption{Schematic explicit construction of the disjoint paths from $\fL_0$ to $\fR_{10}$ for $\ell=3$.}
\label{fig:path_structure}
\end{figure}

To show that this upper bound is strictly achievable, we explicitly construct exactly $\lc$ paths from $\fL_0$ to $\fR_j$ for the case where $j = 3\ell+1$.
For each fixed $k \in \{0, \dots, \ell-1\}$ and $b \in \{1, \dots, \ell-k\}$, we define a unique path $\pi^{(k,b)}$ via the following vertex sequence:
\begin{align*}
k \to k+b \to k+\ell+1 \to 2(\ell+1)-b \to 2(\ell+1)+k.
\end{align*}
See Figure~\ref{fig:path_structure}. 
This construction yields exactly $\sum_{k=0}^{\ell-1} (\ell-k) = \lc$ many distinct paths. To see why the edges of these paths are pairwise disjoint, we classify the paths into two cases.
\begin{enumerate}[{Case}~1.]
    \item ($k_1 \neq k_2$) Paths with different $k$ values are naturally separated because they visit completely distinct sets of vertices at the second node ($k$), the fourth node ($k+\ell+1$), and the final destination node.
    \item ($k_1 = k_2$ but $b_1 \neq b_2$) Paths sharing the same $k$ but having different $b$ values remain disjoint because they branch out at distinct vertices during the third step ($k+b$) and the fifth step ($2\ell-b+3$).
\end{enumerate}
For longer renewal blocks where $j >3\ell+1$, this path structure can be extended to the right periodically by repeating the above construction.
For shorter blocks where $2\ell-1 \le j < 3\ell+1$, the paths can be constructed analogously by connecting their terminal steps directly to $j$. Although some paths may share these initial or final attachment edges, such overlaps are strictly confined to the first and last edges, leaving the remaining edge sets pairwise disjoint.
\end{proof}

\begin{rem}[Alternative proof via Menger's theorem]
The existence of $\lc$ edge-disjoint paths from $\fL_0$ to $\fR_j$ can also be deduced abstractly using the edge version of Menger's theorem (see~\cite[Section 3]{Diestel2017}). Let $A$ be any vertex set containing $\fL_0$ and disjoint from $\fR_j$. We want to find the minimum number of edges crossing from $A$ to its complement $A^c$. 

For each fixed edge length $m \in \{1, \dots, \ell\}$, consider the subgraph consisting only of edges of length $m$. The vertices of this graph can be partitioned into $m$ independent chains with step size $m$. For each of these $m$ chains, the leftmost vertices belong to $\fL_0 \subset A$ and the rightmost vertices belong to $\fR_j \subset A^c$. Therefore, as you move along any single chain from left to right, you must cross the boundary from $A$ to $A^c$ at least once. This guarantees that there are at least $m$ crossing edges of length $m$. 
Summing this over all possible edge lengths $m=1, \dots, \ell$, the total size of any cut separating $\fL_0$ and $\fR_j$ is at least
$
    \sum_{m=1}^{\ell} m = \lc.
$
Since the minimum cut size separating $\fL_0$ and $\fR_j$ is $\lc$, Menger's theorem guarantees the existence of $\lc$ edge-disjoint paths between them. Our explicit construction gives edge-disjoint paths with the same number of edges.
\end{rem}

With the combinatorial path structure, we are ready to prove the moment bounds.
\begin{proof}[Proof of Proposition~\ref{prop:k-moment}]
By translation invariance, set $\rho_1=0$, $\rho_2=\rho$, and $\sfT:=T(0,\rho)$.

\begin{enumeratei}
\item
Assume $0<\gc<k/\lc$. Let
\[
M_\ell := \min_{e\in\cC_{\ell-1}}\go_e,
\]
where $\cC_{\ell-1}$ is the unit cut defined in~\eqref{def:unit_cut}.
For $y>2a$, consider the event 
\[
E_y := \{0 \text{ is a pivot} \} \cap \{\fs_\ell \text{ is a pivot}\} \cap \{M_\ell>y\}.
\]

Observe that the pivot condition at $0$ requires $\go_{0,x}<a$ for $1\le x \le \ell-1$, which violates the crossing condition $\go_{0,x}>2a$ required for $x-1$ to be a pivot. Thus, none of $\{1, \dots, \ell-2\}$ can be pivots. By the same argument, the pivot condition at $\fs_\ell$ rules out pivots in $\{\ell+1, \dots, 2\ell-2\}$. 
Furthermore, on $\{M_\ell>y\}$, we have $\go_{\ell-1,\ell} \ge M_\ell > 2a$. Since a pivot at either $\ell-1$ or $\ell$ requires $\go_{\ell-1,\ell} < a$, neither can be a pivot.

Therefore, the event $E_y$ guarantees that no intermediate pivots exist, implying $E_y \subseteq \{\rho=\fs_\ell\}$ and thus $\sfT = T(0, \fs_\ell) \ge M_\ell > y$.

 The pivot event $\{0 \text{ is a pivot}\}$ depends only on edges with both endpoints in $[-\ell+1, \ell-1]$, $\{\fs_\ell \text{ is a pivot}\}$ depends only on edges within $[\ell, 3\ell-2]$, and $M_\ell$ is determined by edges crossing from $[0, \ell-1]$ to $[\ell, 2\ell-1]$. These three edge sets are mutually disjoint and hence the events defining $E_y$ are independent.

Using this independence and the fact that $|\cC_{\ell-1}| = \lc$, we obtain for all $y>2a$,
\begin{align*}
    \pr(\sfT>y) \ge \pr(E_y) = p_{a,\ell}^2\bar F(y)^{\lc}.
\end{align*}
Therefore,
\begin{align*}
    \E\sfT^k 
    \ge \int_{2a}^\infty ky^{k-1}\pr(\sfT>y)\,dy 
    \ge p_{a,\ell}^2 \int_{2a}^\infty ky^{k-1}\bar F(y)^{\lc}\,dy 
    = \infty,
\end{align*}
since $\bar F(y)=y^{-\gc}L(y)$ and $\gc\lc<k$.

\item  Assume $\gc>k/\lc$. Choose $\eps>0$ such that $k+\eps<\gc\lc$. Fix $j\ge 2\ell-1$.
By Lemma~\ref{lem:path_structure}, there exist $\lc$ paths from $0$ to $j$ whose edge sets are pairwise disjoint except possibly for the first and last edges. On the event $\{\rho=j\}$, the endpoint pivot conditions ensure that the weights of these shared attachment edges are bounded by $a$. Thus, omitting the first and last edges, we obtain $\lc$ completely disjoint paths $\pi_r^{(j)}$ ($1\le r\le \lc$). Let $W_r^{(j)}$ be the passage time of $\pi_r^{(j)}$. Then
\begin{align}\label{eq:T-upper-Zj}
    \sfT\cdot \ind_{\{\rho=j\}} \le Z_j\cdot \ind_{\{\rho=j\}}, \text{ where } Z_j := 2a + \min_{1\le r\le\lc} W_r^{(j)}.
\end{align}
Since each path uses at most $j$ edges, we have $W_r^{(j)} \preceq \sum_{m=1}^{j} \go_{r,m}$, where $\go_{r,m}$ are i.i.d.~copies of $\go$. By the disjointness of the paths, the variables $\{W_r^{(j)}\}_{1\le r\le\lc}$ are independent. Hence, for any large enough $y$,
\begin{equation*}
    \pr(Z_j>y) 
    \le \prod_{r=1}^{\lc} \pr\left(W_r^{(j)} > y-2a\right) 
    \le \big( j\pr(\go > y/2j ) \big)^{\lc}
    \le C j^{\lc(1+\gc)} y^{-\gc\lc} L(y)^{\lc}.
\end{equation*}
for some constant $C>0$. Since $k+\eps < \gc\lc$, we obtain
\begin{equation*}
	    \E Z_j^{k+\eps} = (k+\eps)\int_0^\infty y^{k+\eps-1} \pr(Z_j>y) \,dy \le C' j^{\lc(1+\gc)} \implies \|Z_j\|_{k+\eps} \le C_0 j^A,
\end{equation*}
for some finite constants $C_0, A> 0$ independent of $j$.

We now bound the $k$-th moment of $T$. By Lemma~\ref{lem:path_structure} (i), $T\ind_{\{\rho \le 2\ell-2\}} < 2a$. For the remaining terms, we apply Minkowski's and H\"older's inequalities:
\begin{align*}
    \|\sfT\ind_{\{\rho>2\ell-2\}}\|_k
    \le \sum_{j>2\ell-2} \|\sfT\ind_{\{\rho=j\}}\|_k 
    &\le \sum_{j>2\ell-2} \|Z_j\ind_{\{\rho=j\}}\|_k \\
    &\le \sum_{j>2\ell-2} \|Z_j\|_{k+\eps} \pr(\rho=j)^{\frac{\eps}{k(k+\eps)}} \\
    &\le C_0 \sum_{j>2\ell-2} j^A \pr(\rho=j)^{\frac{\eps}{k(k+\eps)}}.
\end{align*}
By Lemma~\ref{lem:B1-pivot-env}, $\rho$ has an exponential tail, so $\E \sfT^k < \infty$.
When $\E\go^k<\infty$, we can use $\pr(\go>y)\le y^{-k}\E\go^k$ to derive the same result.
\end{enumeratei}
This completes the proof.
\end{proof}

The final prerequisite for our main limit theorems is controlling the boundary residual terms in the pivot decomposition.

\begin{lem}[Endpoint residual estimates]\label{lem:end-residual}
Assume $\ell\ge2$. Recall the residual terms $\sfRT_n$ and $\sfRH_n$ defined in~\eqref{eq:residual-def}.
\begin{enumeratei}
    \item Under Condition~\textup{B.\ref{ass:B1}}, the sequence $\{\sfRT_n\}_{n\ge1}$ is tight, and the sequence $\{\sfRH_n\}_{n\ge1}$ is uniformly bounded in $L^q$ for every $q\ge1$.
    \item Under Condition~\textup{B.\ref{ass:B2}}, both residual sequences have uniformly bounded moments of all orders. That is, for every $q\ge1$,
    \[
        \sup_{n\ge1}\E\left[\big(\sfRT_n\big)^q+\big(\sfRH_n\big)^q\right] < \infty.
    \]
\end{enumeratei}
\end{lem}

\begin{proof}
\begin{enumeratei}
\item The initial terms $T(0,\rho_1)$ and $H(0,\rho_1)$ are finite almost surely and do not depend on $n$, hence they are tight. It remains to prove the terminal residuals are tight.

Let
\begin{align*}
    \sfP_{\gk_n}:=\rho_{\gk_n+1}-\rho_{\gk_n}
\end{align*}
be the pivot gap containing $n$.
We first show that $\{\sfP_{\gk_n}\}_{n\ge1}$ is tight. Recall that sparse trial pivots occur on $\fs_\ell\dZ$, where $\fs_\ell=2\ell-1$, and the trial events are i.i.d.~Bernoulli with parameter $p_{a,\ell}>0$. Let $j(n):=\lfloor n/\fs_\ell\rfloor$. If there is at least one successful trial among
\begin{align*}
    j(n)-M,\ldots,j(n)
\end{align*}
and at least one successful trial among
\begin{align*}
    j(n)+1,\ldots,j(n)+M+1,
\end{align*}
then the pivot gap containing $n$ has length at most $(2M+3)\fs_\ell$. Therefore
\begin{align}\label{eq:endpoint-gap-tight}
    \sup_{n\ge1}
    \pr\left(\sfP_{\gk_n}>(2M+3)\fs_\ell\right)
    \le 2(1-p_{a,\ell})^M,
\end{align}
and hence $\sfP_{\gk_n}$ has uniform $q$-th moment:
\begin{align}\label{eq:endpoint-gap-moments}
    \sup_{n\ge1}\E \sfP_{\gk_n}^q<\infty,
    \quad q\ge1.
\end{align}

For the hop residual, let $\rho_{-1}<0<\rho_1$ be the two consecutive pivots surrounding the origin. Considering the nearest-neighbor paths, we have
\begin{align*}
    H(0,\rho_1)\le \rho_1-\rho_{-1},
    \qquad
    H(\rho_{\gk_n},n)\le \sfP_{\gk_n}.
\end{align*}
Lemma~\ref{lem:B1-pivot-env} (iii) applies also to $\rho_1-\rho_{-1}$, and thus~\eqref{eq:endpoint-gap-moments} gives
$
    \sup_{n\ge1}\E\left[(\sfRH_n)^q\right]<\infty
$
for every $q\ge1$.

We next consider the passage time. Fix $M<\infty$ and set $\fg_{M,\ell}:=(2M+3)\fs_\ell$. Let $\cW_M^T$ denote the worst-case (maximum) passage time between any two vertices within the $\ell$-spread-out line graph on $\{0, 1, \dots, \fg_{M,\ell}\}$, defined as:
\begin{equation*}
    \cW_M^T := \max_{0 \le u < v \le \fg_{M,\ell}} T_{\{0, \dots, \fg_{M,\ell}\}}(u,v),
\end{equation*}
where $T_{\{0, \dots, \fg_{M,\ell}\}}$ is the passage time restricted to paths within this finite graph. Since this graph contains only a finite number of edges, $\cW_M^T<\infty$ a.s. 

On the event $\{\sfP_{\gk_n}\le \fg_{M,\ell}\}$, the terminal residual gap is confined to a length of at most $\fg_{M,\ell}$, which implies $\rho_{\gk_n} \in [n-\fg_{M,\ell}, n]$. By translation invariance, the maximum passage time within this specific interval is stochastically dominated by $\cW_M^T$, yielding:
\begin{equation*}
    T(\rho_{\gk_n},n)\ind_{\{\sfP_{\gk_n}\le \fg_{M,\ell}\}} \preceq \cW_M^T.
\end{equation*}

To prove the tightness of $\{\sfRT_n\}_{n\ge1}$, let $\eps>0$. First, choose $M$ so large that
\begin{equation*}
    \sup_{n\ge1}\pr(\sfP_{\gk_n} > \fg_{M,\ell}) < \eps.
\end{equation*}
Next, since $T(0,\rho_1)$ and $\cW_M^T$ are almost surely finite, we can choose $K$ large enough such that
\begin{equation*}
    \pr(T(0,\rho_1)>K)+\pr(\cW_M^T>K)<\eps.
\end{equation*}
Using the union bound and the stochastic domination on the event $\{\sfP_{\gk_n} \le \fg_{M,\ell}\}$, we obtain
\begin{align*}
    \pr(\sfRT_n>2K)
    &\le \pr(T(0,\rho_1)>K) + \pr\big(T(\rho_{\gk_n},n)>K,\, \sfP_{\gk_n} \le \fg_{M,\ell}\big) + \pr(\sfP_{\gk_n} > \fg_{M,\ell}) \\
    &\le \big[ \pr(T(0,\rho_1)>K) + \pr(\cW_M^T>K) \big] + \eps 
    < \eps + \eps = 2\eps.
\end{align*}
Taking the supremum over $n\ge1$ yields $\sup_{n\ge1}\pr(\sfRT_n>2K)\le 2\eps$, which completes the proof.

\medskip
\item 
For a node $v\in\dZ$, let $G_k(v)$ be the distance between the nearest generalized pivots surrounding $v$. More precisely, define
\[
    \rho^-(v):=\max\{\rho_i:\rho_i<v\},
    \qquad
    \rho^+(v):=\min\{\rho_i:\rho_i\ge v\},
\]
and set
\[
    G_k(v):=\rho^+(v)-\rho^-(v).
\]
The same Bernoulli-trial argument used in~\eqref{eq:endpoint-gap-tight}--\eqref{eq:endpoint-gap-moments}, with $\fs_\ell$ and $p_{a,\ell}$ replaced by $\fs_{\ell,k}$ and $p_{\mathrm{gen},\ell}$, gives
\begin{align}\label{eq:B2-endpoint-gap-moments}
    \sup_{v\in\dZ}\E\big[G_k(v)^q\big]<\infty,
    \quad q\ge1.
\end{align}
The same argument also gives
\begin{align}\label{eq:B2-rho1-moments}
    \E\rho_1^q<\infty,
    \quad q\ge1.
\end{align}

Since $1\le \go\le2$, the nearest-neighbor path gives, for every $a<b$, we have
$
    H(a,b)\le T(a,b)\le 2(b-a).
$
In particular,
\begin{align}\label{eq:B2-endpoint-sum-bound}
    T(a,b)+H(a,b)\le 4(b-a).
\end{align}

On the event $\{\gk_n\ge1\}$, the definition of $\gk_n$ in~\eqref{eq:kappa-def} gives $\rho_{\gk_n} \le n-M_k < \rho_{\gk_n+1}$.
Hence, by the definition of $G_k$,
\begin{align}\label{eq:B2-terminal-gap-bound}
    n-\rho_{\gk_n}
    =
    M_k+\bigl(n-M_k-\rho_{\gk_n}\bigr)
    \le
    M_k+G_k(n-M_k).
\end{align}
Combining~\eqref{eq:B2-endpoint-sum-bound} and
\eqref{eq:B2-terminal-gap-bound}, we obtain
\begin{align}\label{eq:B2-residual-bound-kge1}
    \sfRT_n+\sfRH_n
    &=
    T(0,\rho_1)+H(0,\rho_1)
    +T(\rho_{\gk_n},n)+H(\rho_{\gk_n},n)
    \notag\\
    &\le
    4\rho_1+4\bigl(n-\rho_{\gk_n}\bigr)
    \notag\\
    &\le
    4\bigl(\rho_1+M_k+G_k(n-M_k)\bigr).
\end{align}

On the event $\{\gk_n=0\}$, we have $\rho_1+M_k\ge n$, and again by~\eqref{eq:B2-endpoint-sum-bound},
\begin{align}\label{eq:B2-residual-bound-kzero}
    \sfRT_n+\sfRH_n
    =
    T_n+H_n
    \le 4n
    \le 4(\rho_1+M_k).
\end{align}
The bounds~\eqref{eq:B2-residual-bound-kge1} and
\eqref{eq:B2-residual-bound-kzero}, together with
\eqref{eq:B2-endpoint-gap-moments} and~\eqref{eq:B2-rho1-moments}, imply
that for every $q\ge1$,
\begin{align*}
    \sup_{n\ge1}
    \E\left[
        \big(\sfRT_n\big)^q+\big(\sfRH_n\big)^q
    \right]
    \le
    2\sup_{n\ge1}
    \E\left[
        \big(\sfRT_n+\sfRH_n\big)^q
    \right]
    \le
    C_q
    \left(
        1+\E\rho_1^q+
        \sup_{v\in\dZ}\E G_k(v)^q
    \right)
    <\infty.
\end{align*}
\end{enumeratei}
This completes the proof.
\end{proof}

To synthesize the independent bulk block contributions, we rely on the standard renewal-reward central limit theorem. We state it here in a convenient form tailored to our needs.

\begin{lem}[Renewal--reward CLT]\label{lem:renewal-reward-clt}
Let $\{(X_i,Y_i)\}_{i\ge1}$ be i.i.d.~with $X_i\in\dN$,
\begin{align*}
\E X_1=\mu\in(0,\infty),
\qquad
\E Y_1=\nu\in\dR,
\qquad
\E X_1^2+\E Y_1^2<\infty.
\end{align*}
Define
\begin{align*}
S_m:=\sum_{i=1}^m X_i,
\qquad
R_m:=\sum_{i=1}^m Y_i,
\qquad
N_t:=\max\{m\ge0:S_m\le t\},
\qquad
R(t):=R_{N_t}.
\end{align*}
Then, as $t\to\infty$,
\begin{align}\label{eq:renewal-reward-deterministic-centering}
t^{-1/2}\cdot (R(t)-t\cdot {\nu}/{\mu})
\Longrightarrow
\N(0,\sigma^2),
\end{align}
where $\sigma^2 = \frac{1}{\mu} \var\left(Y_1-\nu/{\mu}\cdot X_1\right)$.
Equivalently,
\begin{align}\label{eq:renewal-reward-mean-centering}
t^{-1/2}\cdot (R(t)-\E R(t))
\Longrightarrow
\N(0,\sigma^2).
\end{align}
\end{lem}

\begin{proof}
This is the standard renewal-reward CLT for i.i.d.~pairs $(X_i,Y_i)$, allowing dependence between $X_i$ and $Y_i$ inside the same pair; see,~\cite[Theorem 3.19]{mitov2014}. The variance is obtained by applying the bivariate CLT to the centered reward $Y_i-\frac{\nu}{\mu}X_i$.
The replacement of the deterministic index by the renewal index $N(t)$ follows from Anscombe's theorem, since $N(t)/t\to1/\mu$ a.s. This gives~\eqref{eq:renewal-reward-deterministic-centering}. The ordinary renewal-reward theorem also gives
$
    \E R(t)=t\cdot {\nu}/{\mu}+O(1),
$
so deterministic centering and mean centering differ by $O(1)/\sqrt t$. This proves~\eqref{eq:renewal-reward-mean-centering}.
\end{proof}

To obtain a non-degenerate Gaussian limit, we must ensure that the asymptotic variances are strictly positive. The following proposition establishes this non-degeneracy for both the passage time and the hop count.

\begin{prop}[Strict positivity of variances]
\label{prop:pos-var}
Assume that $\ell\ge2$ and $\go$ is not deterministic.
\begin{enumeratei}
\item If either $\E \go^2<\infty$, or Assumption~\textup{A.\ref{ass:A2}} holds with $\gc>2/\lc$, then $\gs_{\mathrm{FPP}}^2>0.$ 

\item Without any additional moment assumption, $\gs_{\mathrm{Hop}}^2>0.$
\end{enumeratei}
\end{prop}

\begin{proof}
\begin{enumeratei}
    \item Group every $\ell$ consecutive vertices into one level. Then $\dZ^{(\ell)}$ is essentially a one-dimensional periodic graph.
    
    \noindent\emph{Case 1: Under Condition~\textup{A.\ref{ass:A1}}.} Ahlberg's passage-time central limit theorem guarantees a strictly positive variance for any non-deterministic edge law; see \cite[Theorem~1.2]{a15}. Applying this to the targets $m\ell$ yields a nondegenerate Gaussian limit on the $\sqrt{m}$-scale. Rewriting it on the $\sqrt{m\ell}$-scale and comparing it with Theorem~\ref{thm:LLN-CLT}~\textup{(b)} along $n=m\ell$ gives $\gs_{\mathrm{FPP}}^2>0$.

    \noindent\emph{Case 2: Under Condition~\textup{A.\ref{ass:A2}} with $\gc>2/\lc$.} If $\gs_{\mathrm{FPP}}^2=0$, then \eqref{eq:var-CLT-def} implies
    \begin{align*}
        \sfT_1 = \frac{\mu_{\sfT}}{\mu_{\sfP}}\sfP_1 \quad\text{a.s.}
    \end{align*}
    The right-hand side has an exponential tail by Lemma~\ref{lem:B1-pivot-env}~\textup{(\ref{piv-iii})}, whereas Lemma~\ref{lem:T1-tail} shows that $\sfT_1$ has a heavy tail. This contradiction proves $\gs_{\mathrm{FPP}}^2>0$.

    \item Suppose $\gs_{\mathrm{Hop}}^2=0$. Then,
    \begin{align}\label{eq:hop-degen}
        \sfH_1 = \frac{\mu_{\sfH}}{\mu_{\sfP}}\sfP_1 \quad\text{a.s.}
    \end{align}

    \noindent\emph{Case 1: Under Condition~\textup{B.\ref{ass:B1}}.} 
    Choose a Borel set $J\subset[0,a)$ with $\pr(\go\in J)>0$ such that either $J=\{0\}$ or $J\subset[u,v]$ for some $0<u\le v<2u$. For $b\in\{\ell-1,\ell\}$, define     
    \begin{align*}
        E_d := \{0 \text{ and } b \text{ are ordinary pivots}\} \cap \{\go_e\in J \text{ for every } e\in\cE[0,b]\}.
    \end{align*}
    For both $b \in \{\ell-1, \ell\}$, the endpoint pivot conditions do not conflict and require no heavy edges in $\cE[0,b]$, ensuring $\pr(E_b) > 0$. On this event, intermediate pivots are ruled out, making the single edge $(0,b)$ the unique minimum-hop geodesic. By Lemma~\ref{lem:B1-pivot-env}~\textup{(\ref{piv-ii})}, we have
    \begin{align*}
        \pr\bigl((\sfP_1,\sfH_1)=(\ell-1,1)\bigr) > 0 \text{ and }\pr\bigl((\sfP_1,\sfH_1)=(\ell,1)\bigr) > 0.
    \end{align*}
    This contradicts \eqref{eq:hop-degen} since $1/(\ell-1) \ne 1/\ell$.

    \noindent\emph{Case 2: Under Condition~\textup{B.\ref{ass:B2}}.} Let $Q_x$ denote the event that $x$ is a generalized pivot, and set $b_\star:=k\ell-1$. The events $Q_0\cap Q_\ell$ and $Q_0\cap Q_{b_\star}$ both occur with positive probability.

On $Q_0\cap Q_\ell$, the light edge $(0,\ell)$ jumps over all $z \in (0,\ell)$, which prevents any intermediate vertex $z$ from being a pivot. Since $\go_{0,\ell} < 1 + 1/(k+1) < 2$, this single edge is the unique geodesic, yielding
\begin{align}\label{eq:B2-hop-block-one}
    \pr\bigl((\sfP_1,\sfH_1)=(\ell,1)\bigr) > 0.
\end{align}

Similarly, on $Q_0\cap Q_{b_\star}$, intermediate pivots are ruled out. The explicit path $0 \to \ell \to 2\ell \to \dots \to (k-1)\ell \to b_\star$ consists of $k$ light edges with total weight strictly less than $k+1$. Since any path with $k+1$ or more edges would have a weight of at least $k+1$, the minimum hop count is exactly $k$. Thus,
\begin{align}\label{eq:B2-hop-block-k}
    \pr\bigl((\sfP_1,\sfH_1)=(k\ell-1,k)\bigr) > 0.
\end{align}
Equations \eqref{eq:B2-hop-block-one} and \eqref{eq:B2-hop-block-k} contradict \eqref{eq:hop-degen} since $1/\ell \ne k/(k\ell-1)$. Therefore, $\gs_{\mathrm{Hop}}^2>0$.
\end{enumeratei}
\end{proof}

With all the necessary structural decompositions and integrability bounds, we are now ready to prove the main limit theorems.

\begin{proof}[Proof of Theorem~\ref{thm:LLN-CLT}]
For $i\ge1$, recall that
$\sfP_i:=\rho_{i+1}-\rho_i$,
$\sfT_i:=T(\rho_i,\rho_{i+1})$, and
$\sfH_i:=H(\rho_i,\rho_{i+1})$.  
By Lemma~\ref{lem:B1-pivot-env} under Condition~B.\ref{ass:B1}, and by Lemma~\ref{lem:B2-pivot-env} under Condition~B.\ref{ass:B2}, the sequence $\{(\sfP_i,\sfT_i,\sfH_i)\}_{i\ge1}$ is i.i.d.~Since $0\le\sfH_i\le\sfP_i$, and since the pivot gaps have exponential tails in both cases, we have
\[
    \E\sfP_1^2+\E\sfH_1^2<\infty.
\]

By~\eqref{eq:T-decomp} and~\eqref{eq:H-decomp}, we have
\begin{align}\label{eq:T-H-decomp}
    T_n = \sum_{i=1}^{\gk_n-1}\sfT_i + \sfRT_n,
    \qquad
    H_n = \sum_{i=1}^{\gk_n-1}\sfH_i + \sfRH_n,
\end{align}
where $\sfRT_n$ and $\sfRH_n$ are the endpoint residuals defined in Lemma~\ref{lem:end-residual}. 

For the laws of large numbers, we have $\E\sfT_1<\infty$: under Condition~B.\ref{ass:B1} this follows from Proposition~\ref{prop:k-moment} with $k=1$, while under Condition~B.\ref{ass:B2} it follows from Lemma~\ref{lem:B2-pivot-env} (iv).
The renewal-reward strong law yields
\begin{align*}
    \frac{1}{n}\sum_{i=1}^{\gk_n-1}\sfT_i \to \frac{\mu_{\sfT}}{\mu_{\sfP}}
    \quad\text{a.s.}\quad\text{and}\quad
    \frac{1}{n}\sum_{i=1}^{\gk_n-1}\sfH_i \to \frac{\mu_{\sfH}}{\mu_{\sfP}}
    \quad \text{a.s.}
\end{align*}
By Lemma~\ref{lem:end-residual}, $\{\sfRT_n\}_{n\ge1}$ is tight, which implies $\sfRT_n/n \stackrel{\pr}{\to} 0$, proving weak LLN for $T_n$~\eqref{eq:B1-weak-lln}. 
For the hop-count, Lemma~\ref{lem:end-residual} also gives $\sup_{n\ge1}\E[(\sfRH_n)^2] < \infty$. By Markov's inequality, 
\[
\pr(\sfRH_n > \eps n) \le O(1/n^2)
\]
for any $\eps>0$. Since this is summable over $n$, the Borel--Cantelli lemma implies $\sfRH_n/n \to 0$ a.s. This proves the strong LLN for $H_n$~\eqref{eq:hop-lln}. Moreover, the standard delayed renewal-reward estimate gives 
\begin{equation*}
	\E\left(\sum_{i=1}^{\gk_n-1}\sfH_i\right) = n\frac{\mu_{\sfH}}{\mu_{\sfP}} + O(1).
\end{equation*}
Since $\sup_{n\ge1}\E\sfRH_n < \infty$ by Lemma~\ref{lem:end-residual}, we obtain $\E H_n = n\frac{\mu_{\sfH}}{\mu_{\sfP}} + O(1)$, which implies the mean convergence in~\eqref{eq:hop-lln}.

For the central limit theorems, we have $\E\sfT_1^2<\infty$: under Condition~B.\ref{ass:B1} this follows from Proposition~\ref{prop:k-moment} with $k=2$, while under Condition~B.\ref{ass:B2} it follows from Lemma~\ref{lem:B2-pivot-env} (iv). Define the bulk renewal-reward processes
\begin{align*}
    R_T(t) := \sum_{i=1}^{\sfN(t)}\sfT_i,
    \qquad
    R_H(t) := \sum_{i=1}^{\sfN(t)}\sfH_i,
\end{align*}
where $\sfN(t) := \max\{m\ge0 : \sum_{i=1}^m \sfP_i \le t\}$. Since $\rho_{m+1} = \rho_1 + \sum_{i=1}^m \sfP_i$, the bulk sums in~\eqref{eq:T-H-decomp} can be expressed as
\begin{equation*}
    \sum_{i=1}^{\gk_n-1}\sfT_i = R_T(n-M_k-\rho_1)
    \qquad
    \sum_{i=1}^{\gk_n-1}\sfH_i = R_H(n-M_k-\rho_1).
\end{equation*}
Since $(M_k+\rho_1)/\sqrt n\to0$ a.s., Lemma~\ref{lem:renewal-reward-clt} applies simultaneously to both processes, yielding
\begin{align*}
    \frac{R_T(n-M_k-\rho_1)-n{\mu_{\sfT}}/{\mu_{\sfP}}}{\sqrt n} \Longrightarrow \N(0,\gs_{\mathrm{FPP}}^2),
    \qquad
    \frac{R_H(n-M_k-\rho_1)-n{\mu_{\sfH}}/{\mu_{\sfP}}}{\sqrt n} \Longrightarrow \N(0,\gs_{\mathrm{Hop}}^2).
\end{align*}
Applying Slutsky's theorem to include the $o_{\pr}(\sqrt{n})$ endpoint residuals proves~\eqref{eq:B1-clt}. Since $\E H_n = n\frac{\mu_{\sfH}}{\mu_{\sfP}} + O(1)$, replacing the deterministic centering $n\frac{\mu_{\sfH}}{\mu_{\sfP}}$ with $\E H_n$ proves~\eqref{eq:hop-clt}.

It remains to prove the joint convergence. Fix $(a,b)\in\dR^2$ and consider the scalar renewal reward
\begin{align*}
    Y_i^{a,b}:= a\sfT_i+b\sfH_i.
\end{align*}
The sequence $\{(\sfP_i,Y_i^{a,b})\}_{i\ge1}$ is i.i.d.~and has finite
second moments. Applying Lemma~\ref{lem:renewal-reward-clt} to this
reward, and using $(M_k+\rho_1)/\sqrt n\to0$ almost surely, gives
\begin{align*}
&
\frac{
    aR_T(n-M_k-\rho_1)
    +
    bR_H(n-M_k-\rho_1)
    -
    n(a\mu_{\sfT}+b\mu_{\sfH})/\mu_{\sfP}
}{\sqrt n}
\\
&\hspace{2cm}
\Longrightarrow
\N\left(
    0,\,
    \frac1{\mu_{\sfP}}
    \var\left(
        a\left(
            \sfT_1-\frac{\mu_{\sfT}}{\mu_{\sfP}}\sfP_1
        \right)
        +
        b\left(
            \sfH_1-\frac{\mu_{\sfH}}{\mu_{\sfP}}\sfP_1
        \right)
    \right)
\right).
\end{align*}
By~\eqref{eq:joint-cov-def}, the limiting variance equals $(a,b)\Sigma_\ell(a,b)^T$. Hence, the Cram\'er--Wold theorem yields
\begin{align*}
    \frac1{\sqrt n}
    \begin{pmatrix}
        R_T(n-M_k-\rho_1)
        -
        n\mu_{\sfT}/\mu_{\sfP}
        \\[2mm]
        R_H(n-M_k-\rho_1)
        -
        n\mu_{\sfH}/\mu_{\sfP}
    \end{pmatrix}
    \Longrightarrow
    \N_2(\mvzero,\Sigma_\ell).
\end{align*}
Lemma~\ref{lem:end-residual} gives $\sfRT_n/\sqrt{n}\to0$ and $\sfRH_n/\sqrt{n}\to0$ in probability. 
Combining this with~\eqref{eq:T-H-decomp} and applying Slutsky's theorem proves~\eqref{eq:B1-joint-clt}.
\end{proof}

The proof of the hop-count limit theorem used only the additivity of the observable along forced pivot decompositions and the deterministic bound by the number of edges.  We record the following consequences, which are useful for standard geometric statistics of the selected geodesic.

\begin{cor}[Bounded additive statistics of geodesics]
\label{cor:geo-stat}
Assume that $\ell\ge2$ is fixed and that $\go$ is non-deterministic. Let $\pi_n^\star$ denote the uniquely selected minimum-hop geodesic from $0$ to $n$, and let $\cP$ be the collection of all paths $\pi$ in $(\dZ,\cE^{(\ell)})$. Consider a path statistic $\Phi:\cP \to \dR^d$ that is translation-covariant and additive under concatenation, \ie\
\[
\Phi_{\tau_z\go}(\pi+z)=\Phi_\go(\pi)
\text{ and }\
\Phi(\pi^{(1)}\circ\pi^{(2)}) = \Phi(\pi^{(1)}) + \Phi(\pi^{(2)}).
\]
Assume further that $\norm{\Phi(\pi)} \le C|\pi|$ for some deterministic constant $C<\infty$. For each regeneration block, we define 
 \[
 \Phi_i := \Phi(\pi^\star(\rho_i,\rho_{i+1}))
 \quad 1\le i\le\gk_n-1,
 \quad
 \Phi_n := \Phi(\pi_n^\star),
 \text{ and }\
 \mvmu_\Phi := \E \Phi_1.
 \]
Then the following properties hold:
\begin{enumeratei}
\item \label{item:expectation} \textbf{Expectation.} The exact expectation satisfies
\begin{equation*}
    \E \Phi_n = n \frac{\mvmu_\Phi}{\mu_{\sfP}} + O(1).
\end{equation*}

\item \label{item:J-CLT}\textbf{Joint CLT.} In the finite variance passage-time regime (\ie\ $\E\go^2<\infty$ or Condition~\textup{A.\ref{ass:A2}} holds with $\gc>2/\lc$), 
\begin{equation*}
    n^{-1/2}
    \begin{pmatrix}
        T_n - n \mu_{\sfT}/\mu_{\sfP} \\
        H_n - n \mu_{\sfH}/\mu_{\sfP} \\
        \Phi_n - n \mvmu_\Phi/\mu_{\sfP}
    \end{pmatrix}
    \Rightarrow
    \N_{d+2}(\mvzero, \Sigma_{d+2}),
\end{equation*}
where the $(d+2)\times(d+2)$ asymptotic covariance matrix is given by
\begin{equation*}
    \Sigma_{d+2} := \frac{1}{\mu_{\sfP}} \cov\left( \begin{pmatrix} \sfT_1 \\ \sfH_1 \\ \Phi_1 \end{pmatrix} - \frac{\sfP_1}{\mu_{\sfP}} \begin{pmatrix} \mu_{\sfT} \\ \mu_{\sfH} \\ \mvmu_\Phi \end{pmatrix} \right).
\end{equation*}
Due to \textnormal{(i)}, the expectation $\E \Phi_n$ can be used for centering.

\item \label{item:ratio-limit}\textbf{Empirical ratio limits.} Let $\Psi\ge 0$ be another scalar statistic satisfying the same assumptions with $\mvmu_\Psi := \E\Psi_1 > 0$, and let $\Psi_n := \Psi(\pi_n^\star)$. For definiteness, we set the ratio equal to $0$ on the event $\{\Psi_n = 0\}$. As $n\to\infty$, the empirical ratio converges almost surely:
\begin{equation*}
    \frac{\Phi_n}{\Psi_n} \to  \frac{\mvmu_\Phi}{\mvmu_\Psi} \quad \text{a.s.}
\end{equation*}
Moreover, the ratio satisfies the central limit theorem:
\begin{equation*}
    \sqrt{\Psi_n}\left(\frac{\Phi_n}{\Psi_n} - \frac{\mvmu_\Phi}{\mvmu_\Psi}\right) 
    \Rightarrow 
    \N_d\left(\mvzero,\, \frac{1}{\mvmu_\Psi} \cov\left(\Phi_1 - \frac{\mvmu_\Phi}{\mvmu_\Psi}\Psi_1 \right)\right).
\end{equation*}

\item \label{item:F-CLT}\textbf{Functional CLT.} Let $\mvS_m^\Phi := \sum_{i=1}^m \Phi_i$. In the Skorokhod space $D([0,\infty), \dR^d)$ equipped with the $J_1$ topology, the renewal skeleton process converges to a Brownian motion:
\begin{equation*}
    \left( \frac{1}{\sqrt{n}} \left( \mvS_{\sfN(nt)}^\Phi - nt \frac{\mvmu_\Phi}{\mu_{\sfP}} \right) \right)_{t \ge 0} 
    \Rightarrow 
    \left( B_t^\Phi \right)_{t \ge 0},
\end{equation*}
where $B^\Phi$ is a centered $\dR^d$-valued Brownian motion with covariance function $\cov(B_s^\Phi, B_t^\Phi) = \min\{s,t\} \Sigma_\Phi$ and $\Sigma_\Phi := \frac{1}{\mu_{\sfP}} \cov(\Phi_1 - \frac{\mvmu_\Phi}{\mu_{\sfP}} \sfP_1)$.
\end{enumeratei}
\end{cor}

\begin{proof}
By the pivot decomposition (Lemmas~\ref{lem:B1-pivot-env} and~\ref{lem:B2-pivot-env}), the selected geodesic $\pi_n^\star$ concatenates exactly at the pivot nodes. Thus, it decomposes into independent and identically distributed subpaths $\pi^\star(\rho_i, \rho_{i+1})$ for $1\le i \le \gk_n-1$, along with endpoint residuals, yielding
\begin{equation*}
    \Phi_n = \sum_{i=1}^{\gk_n-1}\Phi_i + \sfR_n^\Phi.
\end{equation*}
Since $\Phi$ is additive under path concatenation and $\norm{\Phi(\pi)} \le C|\pi|$, we have $\norm{\sfR_n^\Phi} \le C\sfRH_n$. By Lemma~\ref{lem:end-residual}, $\sfRH_n$ is tight, which implies that $\sfR_n^\Phi$ is also tight. Thus, $\Phi$ has the exact same regenerative structure as $T_n$ and $H_n$, with $\norm{\Phi_i} \le C\sfH_i \le C\sfP_i$. All second moments needed are finite since the pivot gaps have exponential tails.

The $O(1)$ expectation expansion (\ref{item:expectation}) follows from the standard delayed renewal-reward theorem, as the endpoint residual expectations are uniformly bounded (Lemma~\ref{lem:end-residual}). The joint CLT (\ref{item:J-CLT}) is an immediate extension of Theorem~\ref{thm:LLN-CLT} to $(d+2)$-dimensions, utilizing the exact same Cram\'er--Wold and Slutsky arguments. The ratio limit theorems for $\Phi_n/\Psi_n$ (\ref{item:ratio-limit}) follow directly from the continuous mapping theorem and the delta method. Finally, Donsker's theorem applied to the i.i.d.\ centered sequence $\Phi_i - (\mvmu_\Phi/\mu_P)\sfP_i$ yields the functional Brownian limit on the renewal skeleton (\ref{item:F-CLT}).
\end{proof}

\begin{rem}[Applications to geodesic geometry]\label{rem:geo-examples}
Corollary~\ref{cor:geo-stat} provides a universal framework for analyzing the geometry of geodesics. Several natural geometric observables immediately fall into this class and thus follow essentially for free by exact regeneration:
\begin{enumeratea}
    \item \textbf{Hop count.} Taking $\Phi(\pi)=|\pi|$ recovers the hop-count LLN and CLT from Theorem~\ref{thm:LLN-CLT}.
     \item \textbf{Jump-length count vector.} Let $\Phi(\pi)_d$ be the number of length-$d$ edges in $\pi$, and define $\mvJ_n:=\Phi(\pi_n^\star)$. Introducing the corresponding block variables $\sfJ_i:=\Phi(\pi^\star(\rho_i,\rho_{i+1}))$ and their partial sums $\mvS^{\sfJ}_m:=\sum_{i=1}^m \sfJ_i$, this choice directly yields the joint CLT for $(T_n, H_n, \mvJ_n)$ and the corresponding functional limits for the jump-count processes.
    \item \textbf{Empirical jump measures and lengths.} Applying the ratio limit \textup{(\ref{item:ratio-limit})} with $\Psi(\pi) = |\pi|$ provides the CLT for the empirical jump distribution $\widehat{\mvp}_{n,d} := (\mvJ_n)_d/H_n$ and the empirical mean jump length $\mvL_n := \sum_{d=1}^\ell d\, \widehat{\mvp}_{n,d}$.
    \item \textbf{Bounded edge-marks and backtracking.} For any bounded measurable function $f:\{-\ell,\ldots,-1,1,\ldots,\ell\}\times[0,\infty)\to\dR$, and a path $\pi=(v_0,\ldots,v_{|\pi|})$, the path sum $\Phi(\pi) := \sum_{r=1}^{|\pi|} f(|v_r-v_{r-1}|, \go_{v_{r-1},v_r})$ is a bounded additive statistic. Applying the ratio limit \textup{(\ref{item:ratio-limit})} with $\Psi(\pi) = |\pi|$ immediately gives the a.s.\ convergence and CLT for the empirical average $\Phi_n/H_n$. In particular, the backward-step fraction $\frac{1}{H_n} \sum_{(u,v) \in \pi_n^\star} \ind_{\{v-u < 0\}}$ converges to a limit.
\end{enumeratea}
\end{rem}

%%%%%%%%%%%%%%%%%%%%%%%%%%%%%%%%%%%%%%%%%%%%%%%%%%%%%%%%%%%%%%%%%%%%%%%%%%%%%%%%%%%
\section{Stable Limit Theorem under Assumption A.\ref{ass:A2}}\label{sec:stable}
%%%%%%%%%%%%%%%%%%%%%%%%%%%%%%%%%%%%%%%%%%%%%%%%%%%%%%%%%%%%%%%%%%%%%%%%%%%%%%%%%%%

In this section, we prove the stable limit theorem in the microscopic regime under the heavy-tail assumption. Throughout the section, we denote the tail probability by
\begin{align*}
    \bar F(x):=\pr(\go>x).
\end{align*}
Since Condition~A.\ref{ass:A2} implies Condition~B.\ref{ass:B1}, we can apply the regenerative pivot decomposition established in the preceding section. The primary technical ingredient required here is the exact tail asymptotic for the renewal-block passage time $\sfT_1$. Our analysis proceeds in two steps: first, we exploit the graph's cut geometry to estimate the tail probability of a single block; second, we use this one-block estimate to derive the stable limit.

To study the local geometry, for integers $r<s$, let
\begin{align*}
    \cE[r,s] := \{(p,q):r\le p<q\le s,\ q-p\le\ell\}
\end{align*}
be the edge set of the induced graph on $\{r,r+1,\ldots,s\}$. Recall the unit cut $\cC_u$ defined in~\eqref{def:unit_cut}, which contains exactly $\lc$ edges. We define the minimum edge weight across this cut as
\begin{align*}
    X_u :=\min_{e\in\cC_u}\go_e.
\end{align*}
By the independence of edge weights and the fact that $|\cC_u| = \lc$, we obtain $\pr(X_u>x)=\bar F(x)^{\lc}$. Thus, $\{X_u>x\}$ is precisely the event that every edge crossing the boundary between $u$ and $u+1$ has weight larger than $x$.

To leverage this property, we need to understand how the removal of small sets of edges affects connectivity. The next lemma establishes a rigid geometric dichotomy for any separating edge set of cardinality at most $\lc$: either it fails to completely disconnect the two endpoint windows, or it perfectly coincides with a single unit cut.

\begin{lem}\label{lem:cut-geometry}
Fix $j>2\ell-2$ and set
\begin{align*}
    \fL_0:=\{0,1,\ldots,\ell-1\},
    \qquad
    \fR_j:=\{j-\ell+1,\ldots,j\}.
\end{align*}
For $u\in\{\ell-1,\ldots,j-\ell\}$, let
\begin{align*}
    \cC_u^{(j)}
    :=
    \{(x,y)\in\cE[0,j]:x\le u<y\}.
\end{align*}
Then, for every $S\subseteq\cE[0,j]$, the following hold.
\begin{enumeratei}
    \item If $|S|<\lc$, then deleting $S$ does not disconnect $\fL_0$ from $\fR_j$.

    \item If $|S|=\lc$ and deleting $S$ disconnects $\fL_0$ from $\fR_j$, then
    $S=\cC_u^{(j)}$ for a unique $u\in\{\ell-1,\ldots,j-\ell\}$.  Moreover, for every
    $e\in\cC_u^{(j)}$, there is a simple $\fL_0$-$\fR_j$ path whose unique edge in
    $\cC_u^{(j)}$ is $e$.
\end{enumeratei}
\end{lem}

\begin{proof}
Let $A\subseteq\{0,1,\ldots,j\}$ satisfy
\begin{align*}
    \fL_0\subseteq A,
    \qquad
    A\cap\fR_j=\eset,
\end{align*}
and let $\partial A$ denote the set of edges in $\cE[0,j]$ with one endpoint in $A$ and the other
in $A^c$.

Fix $m\in\{1,\ldots,\ell\}$ and consider only the edges of length $m$.  For each residue class modulo $m$, list its vertices in $\{0,1,\ldots,j\}$ from left to right. Consecutive vertices in this list are joined by an edge of length $m$.  The leftmost vertex lies in $\{0,\ldots,m-1\}\subseteq\fL_0\subseteq A$, whereas the rightmost vertex lies in $\{j-m+1,\ldots,j\}\subseteq\fR_j\subseteq A^c$.  

Thus, in each residue class, at least one length-$m$ edge has one endpoint in $A$ and the other in $A^c$.  Since there are $m$ residue classes modulo $m$, $\partial A$ contains at least $m$ edges of length $m$. 
Summing over $m=1,\ldots,\ell$ gives
\begin{align}\label{eq:mincut-lower-stable}
    |\partial A|
    \ge \sum_{m=1}^{\ell} m
    =  \lc.
\end{align}

Now let $S$ be any edge set separating $\fL_0$ from $\fR_j$.  After deleting the edges in $S$, let $A$ be the set of vertices that can be reached from some vertex of $\fL_0$.  Then $\fL_0\subseteq A$, $A\cap\fR_j=\eset$, and every edge in $\partial A$ must belong to $S$. Hence~\eqref{eq:mincut-lower-stable} implies
\begin{align*}
    |S|\ge|\partial A|\ge\lc.
\end{align*}
On the other hand, each $\cC_u^{(j)}$, with $u\in\{\ell-1,\ldots,j-\ell\}$, separates $\fL_0$ from $\fR_j$ and satisfies
$
    |\cC_u^{(j)}| = \sum_{m=1}^{\ell} m = \lc.
$
Therefore, the minimum cut size is $\lc$.

 Suppose now that $S$ separates $\fL_0$ from $\fR_j$ and $|S|=\lc$. Define $A$ as above. Since $\partial A\subseteq S$,~\eqref{eq:mincut-lower-stable} yields
\begin{align*}
    \lc
    \le |\partial A|
    \le |S|
    = \lc.
\end{align*}
Thus $S=\partial A$.  Equality in~\eqref{eq:mincut-lower-stable} implies, in particular, that $\partial A$ contains exactly one nearest-neighbor edge.  Hence there is a unique $u\in\{0,\ldots,j-1\}$ such that one of $u,u+1$ lies in $A$ and the other lies in $A^c$. Since $0\in A$ and $j\in A^c$, it follows that
$
    A=\{0,1,\ldots,u\}.
$
The conditions $\fL_0\subseteq A$ and $A\cap\fR_j=\eset$ imply $\ell-1\le u\le j-\ell$.
Therefore, $S=\partial A=\cC_u^{(j)}.$
The value of $u$ is unique since $\cC_u^{(j)}$ contains the nearest-neighbor edge $(u,u+1)$.

 Fix $e=(x,y)\in\cC_u^{(j)}$. If $x\in\fL_0$, start at $x$; otherwise, follow nearest-neighbor edges from $\ell-1$ to $x$ inside $\{0,\ldots,u\}$. Traverse the edge $e$. If $y\in\fR_j$, stop at $y$; otherwise, follow nearest-neighbor edges from $y$ to $j-\ell+1$ inside $\{u+1,\ldots,j\}$.  The resulting path is simple, and $e$ is its only edge crossing from $\{0,\ldots,u\}$ to $\{u+1,\ldots,j\}$. Hence, it uses no other edge of $\cC_u^{(j)}$.
\end{proof}

The previous lemma shows a geometric bottleneck: the only way to separate the two ends of a block by removing just $\lc$ edges is to remove an entire unit cut. Consequently, an exceptionally large passage time across a block is primarily caused by a single rare event: all $\lc$ edges of some interior unit cut being simultaneously very heavy. To quantify the frequency of such local events within a randomly sized renewal block, we need a standard ergodic identity.

\begin{lem}\label{lem:ergodic}
Let $\rho_1<\rho_2<\cdots$ be the sequence of pivot nodes. For any non-negative measurable function $f:\dR^{\cE^{(\ell)}}\to\dR_+$ defined on the edge environment $\mvgo$, we have
\begin{align}\label{eq:ergodic}
    \E\left( \sum_{u=\rho_1}^{\rho_2-1} f(\theta_u\mvgo) \right) 
    = \mu_{\sfP}\,\E[f(\mvgo)],
\end{align}
where $\theta_u$ is the spatial shift operator by $u$, and $\mu_{\sfP} = \E(\rho_2-\rho_1)$.
\end{lem}

\begin{proof}
By the spatial ergodic theorem for the stationary sequence $\{f(\theta_u\mvgo)\}_{u\in\dZ}$, we have
\begin{align*}
    \lim_{n\to\infty} \frac{1}{\rho_{n+1}-\rho_1} \sum_{u=\rho_1}^{\rho_{n+1}-1} f(\theta_u\mvgo) = \E[f(\mvgo)] \quad \text{a.s.}
\end{align*}
On the other hand, since the blocks are i.i.d., the strong law of large numbers implies
\begin{align*}
    \lim_{n\to\infty} \frac{1}{n} \sum_{i=1}^n \sum_{u=\rho_i}^{\rho_{i+1}-1} f(\theta_u\mvgo) = \E\left( \sum_{u=\rho_1}^{\rho_2-1} f(\theta_u\mvgo) \right) \quad \text{a.s.}
\end{align*}
Since $(\rho_{n+1}-\rho_1)/n \to \mu_{\sfP}$ a.s., taking the ratio of the two limits yields \eqref{eq:ergodic}.
\end{proof}

With this identity, we can now establish the exact tail asymptotic for the block passage time.

\begin{lem}[Exact one-block tail]\label{lem:T1-tail}
Assume Condition~\textup{A.\ref{ass:A2}}, and let $\rho_1<\rho_2$ be consecutive pivot nodes. Then
\begin{align}\label{eq:Ttail}
    \bar F(x)^{-\lc}\cdot \pr\bigl(T(\rho_1,\rho_2)>x\bigr)
    \to
    \E(\rho_2-\rho_1)\,
    \text{ as }x\to\infty.
\end{align}
\end{lem}

\begin{proof}
For $x>0$ and $i\ge1$, define
\begin{align*}
    N_x^{(i)}:= \sum_{u=\rho_i}^{\rho_{i+1}-1}\ind_{\{X_u>x\}}.
\end{align*}
If $x>a$, then the pivot conditions imply $X_u<a$ for the $(\ell-1)$ many nearest neighbor endpoint. Thus only $ \rho_i+\ell-1\le u\le \rho_{i+1}-\ell$ can contribute to $N_x^{(i)}$. 
Every such cut is contained in the renewal-block $[\rho_i,\rho_{i+1}]$, so Lemma~\ref{lem:B1-pivot-env} implies that $\{N_x^{(i)}\}_{i\ge1}$ is i.i.d.~Moreover, $N_x^{(i)}\le \rho_{i+1}-\rho_i$, and hence these variables are integrable.

Since
\begin{align*}
    \sum_{i=1}^{m}N_x^{(i)} = \sum_{u=\rho_1}^{\rho_{m+1}-1}\ind_{\{X_u>x\}},
\end{align*}
the strong law for the renewal blocks and the ergodic theorem for the stationary finite-range sequence $\{\ind_{\{X_u>x\}}\}_{u\in\dZ}$ give
\begin{align*}
    \frac1m\sum_{i=1}^{m}N_x^{(i)}
    \to
    \E N_x^{(1)},
    \qquad
    \frac{\rho_{m+1}-\rho_1}{m}
    \to
    \mu_{\sfP},
    \qquad
    \frac{1}{\rho_{m+1}-\rho_1}
    \sum_{u=\rho_1}^{\rho_{m+1}-1}\ind_{\{X_u>x\}}
    \to
    \bar F(x)^{\lc}
\end{align*}
almost surely, where $\mu_{\sfP}=\E(\rho_2-\rho_1).$ Therefore, for every $x>a$,
\begin{align}\label{eq:ENx}
    \E N_x^{(1)} = \mu_{\sfP}\bar F(x)^{\lc}.
\end{align}

For the remainder of the proof, translate the first block by $-\rho_1$ and simply write
\begin{align*}
    \rho:=\rho_2-\rho_1,
    \qquad
    \sfT:=T(0,\rho)=T(\rho_1,\rho_2),
    \qquad
    N_x:=N_x^{(1)}.
\end{align*}
The geometric-trial estimate in Lemma~\ref{lem:B1-pivot-env} gives an exponential tail for $\rho$. Since $\bar F(x)^{\lc}$ is regularly varying, there exists a constant $B_0>0$ such that
\begin{align}\label{eq:rho-log-trunc-stable}
    \pr(\rho>B_0\log x) + \E\left[\rho^2\ind_{\{\rho>B_0\log x\}}\right]
    = o\bigl(\bar F(x)^{\lc}\bigr).
\end{align}
For every $m\ge1$, the cuts $\cC_0$ and $\cC_m$ are distinct and each has cardinality $\lc$. Hence $|\cC_0\cup\cC_m| \ge \lc+1$.
Since
\begin{align*}
    \binom{N_x}{2}
    \ind_{\{\rho\le B_0\log x\}}
    \le \sum_{u=0}^{\rho-1}
    \sum_{m=1}^{\lceil B_0\log x\rceil}
    \ind_{\{X_u>x,\,X_{u+m}>x\}},
\end{align*}
Applying Lemma~\ref{lem:ergodic} with $f(\mvgo) = \sum_{m=1}^{\lceil B_0\log x\rceil} \ind_{\{X_0>x,\,X_m>x\}}$, taking the expectation yields
\begin{align*}
    \E\left[\binom{N_x}{2} \ind_{\{\rho\le B_0\log x\}}\right]
    \le \E\left[ \sum_{u=0}^{\rho-1} \sum_{m=1}^{\lceil B_0\log x\rceil} \ind_{\{X_u>x,\,X_{u+m}>x\}} \right] 
    = \mu_{\sfP}\sum_{m=1}^{\lceil B_0\log x\rceil}  \pr(X_0>x,\,X_m>x).
\end{align*}
By independence of the edge weights and $|\cC_0\cup\cC_m| \ge \lc+1$,
\begin{align*}
    \E\binom{N_x}{2}
    &\le \E\left[\rho^2\ind_{\{\rho>B_0\log x\}}\right] + \mu_{\sfP}\sum_{m=1}^{\lceil B_0\log x\rceil} \pr(X_0>x,\,X_m>x)\notag\\
    &\le o\bigl(\bar F(x)^{\lc}\bigr) + C(\log x)\bar F(x)^{\lc+1}
     = o\bigl(\bar F(x)^{\lc}\bigr).
\end{align*}

We first prove the lower bound. If $N_x\ge1$, every path from $0$ to $\rho$ crosses an $x$-heavy unit cut, and therefore $\sfT>x$. Hence
\begin{align}\label{eq:Ttail-lower}
    \pr(\sfT>x)
    \ge \pr(N_x\ge1)
    \ge \E N_x-\E\binom{N_x}{2}
    = \mu_{\sfP}\bar F(x)^{\lc} + o\bigl(\bar F(x)^{\lc}\bigr).
\end{align}

For the upper bound, fix $\eps\in(0,1/2)$ and define
\begin{align*}
    \eta_x := \frac{\eps x}{4(B_0\log x+2)},
    \qquad
    S_x := \{e\in\cE[0,\rho]:\go_e>\eta_x\}.
\end{align*}
We claim that, for all sufficiently large $x$,
\begin{align}\label{eq:Ttail-upper-inclusion}
    \{\sfT>x,\ \rho\le B_0\log x,\ N_{(1-\eps)x}=0\}
    \subseteq
    \{|S_x|\ge\lc+1,\ \rho\le B_0\log x\}.
\end{align}
Indeed, work on $\{\rho\le B_0\log x\}$. If $\rho\le2\ell-2$, then
Lemma~\ref{lem:path_structure} (i) gives $\sfT<2a$, so this case does not happen for large $x$.
Assume that $\rho>2\ell-2$.
If $S_x$ does not separate $\fL_0$ from $\fR_\rho$, there is a simple path between the two windows avoiding $S_x$. Attaching its endpoints to $0$ and $\rho$ costs at most $2a$, and hence
\begin{align*}
    \sfT \ind_{\{\rho\le B_0\log x\}} \le 2a+(\rho+1)\eta_x < x
\end{align*}
for all sufficiently large $x$, a contradiction. Thus, on $\{\sfT>x\}$, the set $S_x$ separates $\fL_0$ from $\fR_\rho$ and $|S_x|\ge\lc$ by Lemma~\ref{lem:cut-geometry}.

Suppose that $|S_x|=\lc$. The same lemma yields $S_x=\cC_u^{(\rho)}$ for some $u\in\{\ell-1,\ldots,\rho-\ell\}$. Since this is an interior cut,
$\cC_u^{(\rho)}=\cC_u$. On $\{N_{(1-\eps)x}=0\}$, choose
$e\in\cC_u$ with $\go_e\le(1-\eps)x$. By
Lemma~\ref{lem:cut-geometry}, there is a simple $\fL_0$-$\fR_\rho$ path whose unique edge
in $S_x$ is $e$. Every other edge on this path has weight at most $\eta_x$, so
\begin{align*}
    \sfT \ind_{\{\rho\le B_0\log x,\ N_{(1-\eps)x}=0\}}
    \le
    2a+(1-\eps)x+(\rho+1)\eta_x
    <
    x
\end{align*}
for all sufficiently large $x$, a contradiction. This proves
\eqref{eq:Ttail-upper-inclusion}.

To bound the probability on the right-hand side of~\eqref{eq:Ttail-upper-inclusion}, we define $Y_u$ such that $u$ is a pivot node and the renewal block starting at $u$ is ``bad'':
\[ Y_u := \ind_{\{u \text{ is a pivot node}\}} \cdot \ind_{\left\{ \rho(u) \le B_0\log x, \, \sum_{e \in \cE[u, u+\rho(u)]} \ind_{\{\go_e > \eta_x\}} \ge \lc+1 \right\}},
\]
where $\rho(u)$ denotes the distance from $u$ to the next pivot node. In particular, for $u=0$,
\[
Y_0 = \ind_{\{0 \text{ is a pivot node}\}} \cdot \ind_{\left\{ \rho \le B_0\log x, \, |S_x| \ge \lc+1 \right\}}.
\]
Observe that the only pivot node is $\rho_1$ in the interval $[\rho_1, \rho_2-1]$, so we have
$
    \sum_{u=\rho_1}^{\rho_2-1} Y_u = Y_{\rho_1}.
$
Since $\E[Y_{\rho_1}]$ is the probability that the first block is bad, applying Lemma~\ref{lem:ergodic} with $f(\mvgo) = Y_0$ gives
\begin{align}\label{eq:id-stable-lem}
    \pr(|S_x|\ge\lc+1,\ \rho\le B_0\log x) 
    = \E Y_{\rho_1} 
    = \E\left( \sum_{u=\rho_1}^{\rho_2-1} Y_u \right)
    = \mu_{\sfP}\, \E Y_0.
\end{align}

On the other hand, for $Y_0=1$ to occur, the  interval $[0, \lceil B_0\log x \rceil]$ to contain those $\lc+1$ heavy edges. Thus, we have 
\begin{align*}
    Y_0 \le \ind_{\left\{ \sum_{e\in\cE[0,\lceil B_0\log x\rceil]} \ind_{\{\go_e>\eta_x\}} \ge\lc+1 \right\} } .
\end{align*}

Taking the expectation of this inequality and substituting it back into \eqref{eq:id-stable-lem}, we obtain
\begin{align}\label{eq:Ttail-negligible}
    \pr(|S_x|\ge\lc+1,\ \rho\le B_0\log x) 
    = \mu_{\sfP} \E[Y_0] 
    &\le \mu_{\sfP}\, \pr\left( \sum_{e\in\cE[0,\lceil B_0\log x\rceil]} \ind_{\{\go_e>\eta_x\}} \ge\lc+1 \right) \notag\\
    &\le C(\log x)^{\lc+1}\bar F(\eta_x)^{\lc+1}.
\end{align}

Since $\bar F(\eta_x)\le C(\log x)^{\gc+1}\bar F(x)$ for all sufficiently large $x$,
the right-hand side of~\eqref{eq:Ttail-negligible} is
$o(\bar F(x)^{\lc})$.

Combining~\eqref{eq:rho-log-trunc-stable},~\eqref{eq:Ttail-upper-inclusion}, and~\eqref{eq:Ttail-negligible}, we obtain
\begin{align*}
    \pr(\sfT>x)
    \le \pr(N_{(1-\eps)x}\ge1) + o\bigl(\bar F(x)^{\lc}\bigr)
    &\le \E N_{(1-\eps)x} + o\bigl(\bar F(x)^{\lc}\bigr) \\
    &= \mu_{\sfP}\bar F((1-\eps)x)^{\lc} + o\bigl(\bar F(x)^{\lc}\bigr).
\end{align*}
Therefore
\begin{align*}
    \limsup_{x\to\infty} {\bar F(x)^{-\lc}}\cdot {\pr(\sfT>x)}
    \le \mu_{\sfP}(1-\eps)^{-\gc\lc}.
\end{align*}
Letting $\eps\downarrow0$ and using~\eqref{eq:Ttail-lower} proves~\eqref{eq:Ttail}.
\end{proof}

Lemma~\ref{lem:T1-tail} shows that the regeneration reward $\sfT_1$ is regularly varying with tail index $\gc\lc$ and tail constant $\mu_{\sfP}=\E\sfP_1$. Therefore, the stable limit is now a renewal-reward consequence: Apply the classical stable limit theorem to the i.i.d.~block rewards, and then replace the deterministic number of blocks by the random renewal count, with the negligible endpoint residuals.

\begin{thm}[Theorem~\ref{thm:micro}-(\ref{thm:sl})]\label{thm:micro-stable}
Suppose Assumption~\textup{A.\ref{ass:A2}} holds and $0<\ga:=\gc\lc<2$.
Let $m_n:=\left\lfloor n/\mu_{\sfP} \right\rfloor$ and define
\begin{align}\label{eq:stable-law-const}
    a_n :=\inf\left\{x>0:m_n\pr(\sfT_1>x)\le1\right\},
    \quad\text{and}\quad
    b_n :=m_n\,\E\left[\sfT_1\ind_{\{\sfT_1\le a_n\}}\right],
\end{align}
where $\sfT$ is the renewal-block passage time defined in~\eqref{eq:block-variable-notation}.
Then
\begin{align*}
    \frac1{a_n}\cdot (T_n-b_n)
    \Longrightarrow
    Z_\ga,
\end{align*}
where $Z_\ga$ is the totally right-skewed $\ga$-stable random variable characterized by
\begin{align}\label{eq:stable-law-char}
    \E\left[e^{itZ_\ga}\right]
    =
    \exp\left(
        \int_0^\infty
        \left(
            e^{itx} - 1 - itx\ind_{\{x\le1\}}
        \right)
        \ga x^{-\ga-1}\,dx
    \right),
    \quad
    t\in\dR.
\end{align}
\end{thm}

\begin{proof}
By Lemma~\ref{lem:B1-pivot-env},
\begin{align*}
    \{(\sfP_i,\sfT_i)\}_{i\ge1}
    =
    \{(\rho_{i+1}-\rho_i,\ T(\rho_i,\rho_{i+1}))\}_{i\ge1}
\end{align*}
is i.i.d., and $\sfP_i$ has finite moments of all orders.  Lemma~\ref{lem:T1-tail} shows that $\sfT_i$ is regularly varying with tail index $\ga=\gc\lc\in(0,2)$.  Therefore, the standard stable limit theorem gives
\begin{align}\label{eq:deterministic-stable-sum}
    \frac1{a_n}\left(\sum_{i=1}^{m_n}\sfT_i-b_n\right)
    \Longrightarrow
    Z_\ga.
\end{align}

We next replace $m_n$ by the renewal index $M_n:=\gk_n-1$,
    where $\gk_n:=\#\{i:\rho_i\le n\}.$
Since $\E\sfP_1^2<\infty$ and $\E\sfP_1=\mu_{\sfP}$, the fluctuation for renewal counting estimate gives
\begin{align}\label{eq:renewal-index-tight}
    M_n-m_n=O_{\pr}(\sqrt n).
\end{align}
Choose $\beta>0$ such that
\begin{align*}
    \frac12<\beta<1\wedge\frac1\ga,
    \quad\text{and set}\quad
    r_n:=\lceil n^\beta\rceil.
\end{align*}
Then~\eqref{eq:renewal-index-tight} implies
$\pr(|M_n-m_n|>r_n)\to 0$. On $\{|M_n-m_n|\le r_n\}$,
\begin{align*}
    \left|
    \sum_{i=1}^{M_n}\sfT_i - \sum_{i=1}^{m_n}\sfT_i
    \right|
    \le
    \sum_{i=m_n-r_n+1}^{m_n+r_n}\sfT_i.
\end{align*}
Since $m_n\pr(\sfT_1>a_n)\to1$, for every fixed $\gd>0$,
\begin{align*}
    r_n\pr(\sfT_1>\gd a_n)
    \to
    0.
\end{align*}
Moreover,
\begin{align}\label{eq:stable-trunc-vanishing}
    \frac{r_n}{a_n}
    \E\left[\sfT_1\ind_{\{\sfT_1\le\gd a_n\}}\right]
    \to
    0.
\end{align}
Indeed, for $0<\ga<1$, Karamata's theorem gives
\begin{align*}
    \E\left[\sfT_1\ind_{\{\sfT_1\le x\}}\right]
    \asymp
    \frac{\ga}{1-\ga}\,x\pr(\sfT_1>x),
\end{align*}
so~\eqref{eq:stable-trunc-vanishing} follows from the tail estimate. If
$1<\ga<2$, then $\E\sfT_1<\infty$ and $r_n/a_n\to0$. If $\ga=1$, the truncated mean is slowly varying, while $r_n/a_n$ is regularly varying with negative index, so~\eqref{eq:stable-trunc-vanishing} holds.
Thus, for every $\eps>0$,
\begin{align*}
    \pr\left(\sum_{i=1}^{2r_n}\sfT_i>\eps a_n\right)
    &\le
    2r_n\pr(\sfT_1>\gd a_n)
    +
    \frac{2r_n}{\eps a_n}
    \E\left[\sfT_1\ind_{\{\sfT_1\le\gd a_n\}}\right]
    \to
    0.
\end{align*}
Consequently,
\begin{align}\label{eq:random-deterministic-repl}
    \frac1{a_n}
    \left|
    \sum_{i=1}^{M_n}\sfT_i
    -
    \sum_{i=1}^{m_n}\sfT_i
    \right|
    \stackrel{\pr}{\to}
    0.
\end{align}
Combining~\eqref{eq:deterministic-stable-sum} and
\eqref{eq:random-deterministic-repl},
\begin{align}\label{eq:bulk-random-stable}
    \frac1{a_n}\left(\sum_{i=1}^{M_n}\sfT_i-b_n\right)
    \Longrightarrow
    Z_\ga.
\end{align}

It remains to control the endpoint terms.  The pivot decomposition~\eqref{eq:T-decomp} gives $T_n = T(0,\rho_1)+ \sum_{i=1}^{M_n}\sfT_i + T(\rho_{\gk_n},n)$.
Lemma~\ref{lem:end-residual} implies that the initial gap $\rho_1$ and the terminal residual gap $n-\rho_{\gk_n}$ form tight families. Fix $A\in\dN$. On
$
    \{\rho_1\le A,\ n-\rho_{\gk_n}\le A\},
$
by considering nearest-neighbor paths, we have
\begin{align*}
	T(0,\rho_1)+T(\rho_{\gk_n},n) 
	\le \sum_{i=0}^{A-1}\go_{i,i+1} + \sum_{j=n-A}^{n-1} \go_{j,j+1} 
	\preceq \sum_{i=1}^{2A} \go_i
\end{align*}
which is finite almost surely. Since $a_n\to\infty$,
for every $\eps>0$,
\begin{align*}
    \lim_{n\to\infty}
    \pr\left(
      T(0,\rho_1)+T(\rho_{\gk_n},n)>\eps a_n,\
      \rho_1\le A,\ n-\rho_{\gk_n}\le A
    \right)
    =0.
\end{align*}
Letting $A\to\infty$ proves
$(T(0,\rho_1)+T(\rho_{\gk_n},n))/a_n
    \stackrel{\pr}{\to}
    0.
$
Slutsky's theorem and~\eqref{eq:bulk-random-stable} complete the proof.
\end{proof}

%%%%%%%%%%%%%%%%%%%%%%%%%%%%%%%%%%%%%%%%%%%%%%%%%%%%%%%%%%%%%%%%%%%%%%%%%%%%%%%
\section{Proofs for Auxiliary Lemmas}
\label{sec:aux}

In this section, we provide the deferred proofs for the structural lemmas underlying the pivot decompositions introduced in Section~\ref{sec:pivot}. These proofs verify the geometric localization of geodesics and the independence of the environment blocks separated by pivot nodes, which are the cornerstones of our renewal-reward limit theorems for both Condition~B.\ref{ass:B1} and Condition~B.\ref{ass:B2}.

\begin{proof}[Proof of Lemma~\ref{lem:B1-pivot-env}]
\begin{enumeratei}
	\item Let $\pi^\star$ be a geodesic from $\rho_i$ to $\rho_j$. By Lemma~\ref{lem:pivot}, any path starting before a pivot node or ending after one must pass through it. Since all edge weights are non-negative, any excursion outside the interval $[\rho_i, \rho_j]$ would create a cycle, which would increase the total weight. Hence $\pi^\star$ cannot visit a vertex smaller than $\rho_i$. 
 		Therefore $\pi^\star$ is contained in $\cE[\rho_i,\rho_j]$, and the measurability of $T(\rho_i,\rho_j)$ follows. Since $\pi^\star$ can be chosen to have the minimum number of edges among all geodesics, the same argument gives the stated localization for the hop-count geodesic.

\medskip
	\item For a vertex $v$, write $P_v$ for the event that $v$ is a pivot node, and decompose
\begin{align*}
P_v^-&:=
\{\go_{x,v}<a:\ v-\ell+1\le x\le v-1\},\\
P_v^+&:=
\{\go_{v,y}<a:\ v+1\le y\le v+\ell-1\},\\
P_v^\times&:=
\{\go_{x,y}>2a:\ x<v<y,\ y-x\le\ell\}.
\end{align*}
Then
\begin{align*}
    P_v=P_v^-\cap P_v^+\cap P_v^\times,
\end{align*}
and the three events depend on disjoint edge sets. Hence, conditioning on $P_v$ does not create dependence between the environment strictly to the left of $v$ and the environment strictly to the right of $v$.

By the definition of pivot, on $P_v$, none of
\begin{align*}
    v+1,\ldots,v+\ell-2
\end{align*}
can be a pivot. Indeed, if $w=v+k$ with $1\le k\le\ell-2$, then $P_v$ gives $\go_{v-1,w}>2a$, while $P_w$ would require $\go_{v-1,w}<a$. 

Thus, after finding a pivot at $v$, the search for the next block relies entirely on the unexplored right side, using $P_v^+$ as the initial condition.

Now, consider two consecutive pivots, $u$ and $v$. As shown earlier, once we finish exploring the block from $u$ to $v$, the remaining environment to the right of $v$ becomes completely independent of this past block. Operating only under the starting condition $P_v^+$, this new environment behaves the same way as the environment did right after $u$.
 Translation invariance shows that the blocks in~\eqref{eq:pivot-iid-str} are independent and identically distributed. The i.i.d.~statements involving $T(\rho_i,\rho_{i+1})$ and $H(\rho_i,\rho_{i+1})$ follow from the localization proved in (i).

\medskip
	\item It remains to prove the gap domination. Consider only the sites $v + k \fs_\ell$ for $k \ge 1$.
	The pivot windows of these trial sites are pairwise disjoint, so the corresponding events 
	\begin{align*}
		\{v + k \fs_\ell \text{ is a pivot} \}_{k\ge1}
	\end{align*}
	 are i.i.d.~Bernoulli random variables with parameter $p_{a,\ell}$. The next actual pivot node occurs no later than the first successful trial on $v + k \fs_\ell$, which geometrically dominates the gap length as stated in~\eqref{eq:pivot-gap-domin}. Hence $\E(\rho_2-\rho_1)<\infty$, and~\eqref{eq:pivot-renewal-slln} directly follows from the renewal strong law of large numbers.
\end{enumeratei}
\end{proof}

Next, we look at generalized pivots under Condition~B.\ref{ass:B2} and show they share similar properties. Before proving the geodesic forcing property (Lemma~\ref{lem:B2-kpivot-forcing}), we need to see what happens if a path tries to detour the pivot chain. The following geometric lemma is the key counting fact and is an easy consequence of the $\ell$ restriction on the distance between the edge endpoints.
\begin{lem}[Crossing count]\label{lem:B2-kpivot-count}
Let
\begin{align*}
    L_j:=x+j\ell,
    \quad
    j=0,1,\ldots,m-1.
\end{align*}
Assume that every edge crossing over one of the levels $L_0,\ldots,L_{m-1}$ is heavy. Then any path that starts to the left of $L_0$, ends to the right of $L_{m-1}$, and avoids all vertices
\begin{align*}
    L_0,\ldots,L_{m-1},
\end{align*}
must contain at least $m$ heavy edges.
\end{lem}

\begin{proof}
The path must cross each level $L_j$. Since every edge crossing one of these levels is heavy by assumption, it suffices to show that a single edge cannot cross two distinct levels.

Suppose an edge crosses both $L_j$ and $L_{j+1}$. Then its endpoints must lie on opposite sides of both levels, and hence its length is strictly greater than
$
    L_{j+1}-L_j=\ell.
$
This contradicts the assumption that all edges have endpoint distance at most $\ell$. Therefore, each edge crosses at most one level. Since the path crosses all $m$ levels, it must contain at least $m$ distinct heavy crossing edges.
\end{proof}

\begin{proof}[Proof of Lemma~\ref{lem:B2-kpivot-forcing}]
Suppose, for contradiction, that there exists a geodesic $\pi^\star$ from $u$ to $w$ that does not pass through $x$. Let $y$ be the last vertex of $\pi^\star$ strictly to the left of $x$. Then we have $x-\ell<y<x$.

\medskip
\noindent\textbf{Case 1.}
Suppose that $\pi^\star$ first meets the pivot chain at
$
    x+i\ell
$
for some $1\le i<k$. Thus $\pi^\star$ avoids
\begin{align*}
    x,x+\ell,\ldots,x+(i-1)\ell
\end{align*}
but passes through $x+i\ell$. By Lemma~\ref{lem:B2-kpivot-count}, the segment of $\pi^\star$ from $y$ to $x+i\ell$ contains at least $i$ heavy edges. The final edge entering $x+i\ell$ is light, since it is incident to a pivot-chain vertex. Hence, this segment has a weight of at least
$
    i(1+2/k)+1.
$
Replace this segment with the corridor
\begin{align*}
    y\to x\to x+\ell\to\cdots\to x+i\ell .
\end{align*}
This corridor contains exactly $(i+1)$ light edges and therefore has a weight at most
$
    (i+1)(1+1/(k+1)).
$
The decrease in total weight is at least
\begin{align*}
    i(1+2/k)+1-(i+1)(1+1/(k+1))
    = i(2/k-1/(k+1))-1/(k+1)
    \ge 2/k-2/(k+1) > 0.
\end{align*}
Thus, the replacement strictly lowers the passage time, contradicting the geodesicity of $\pi^\star$.

\medskip
\noindent\textbf{Case 2.}
Suppose that $\pi^\star$ avoids every pivot-chain vertex in the set 
$
    \{x,x+\ell,\ldots,x+(k-1)\ell\}.
$
Let $z$ be the first vertex of $\pi^\star$ to the right of $x+(k-1)\ell$. Applying
Lemma~\ref{lem:B2-kpivot-count} with $m=k$, the segment of $\pi^\star$ from $y$ to $z$ contains
at least $k$ heavy edges. Hence, its weight is strictly bigger than
$
    k(1+2/k)=k+2.
$
Replace this segment by
\begin{align*}
    y\to x\to x+\ell\to\cdots\to x+(k-1)\ell\to z.
\end{align*}
This corridor contains exactly $(k+1)$ many light edges and therefore has a weight strictly smaller than
$
    (k+1)(1+1/(k+1))=k+2.
$
The decrease in weight is strictly positive.
Again, we obtain a strictly shorter path, contradicting the geodesicity.

Both cases lead to contradictions. Thus, every geodesic from $u$ to $w$ must pass through $x$.
\end{proof}

\begin{proof}[Proof of Lemma~\ref{lem:B2-pivot-env}]
Under Condition~\textup{B.\ref{ass:B2}}, define $M_k:=(k-1)\ell$.
\begin{enumeratei}
\item Let $r:=\rho_i$ and $s:=\rho_j$, and choose a geodesic $\pi^\star=(v_0,\ldots,v_m)$ from $r$ to $s$ with the minimum number of edges among all such geodesics. We will show that $\pi^\star$ is strictly contained in $[r,s]$.

Suppose first that $\pi^\star$ visits a vertex strictly smaller than $r$. Let $p$ be the first index with $v_p<r$, and let $q>p$ be the first later index with $v_q\ge r$. Since all edge weights are strictly positive, a geodesic cannot contain loops; therefore, the path cannot return to $r$ once it has left it, which implies $v_q>r$.

Thus, the edge $(v_{q-1},v_q)$ crosses over the generalized pivot vertex $r$, and must therefore be heavy:
\begin{align*}
    \go_{v_{q-1},v_q}>1+2/k.
\end{align*}
Moreover, since $v_{q-1} < r < v_q$, we have $v_q-r<v_q-v_{q-1}\le \ell$, so the edge $(r,v_q)$ is valid. Since this edge is connected to the pivot chain $S_r$, it is light:
\begin{align*}
    \go_{r,v_q}<1+1/(k+1).
\end{align*}
The subpath from $v_0=r$ to $v_{q-1}$ contains at least one edge, meaning its accumulated weight is at least $1$. If we replace the entire segment of the path from $r$ to $v_q$ with the single edge $(r,v_q)$, the total passage time strictly decreases by more than
\begin{align*}
    (1+1+2/k) - (1+1/(k+1))
    = 1+2/k- 1/(k+1) >0,
\end{align*}
which contradicts the geodesicity of $\pi^\star$. Hence, $\pi^\star$ never visits a vertex smaller than $r$.

By a symmetric argument, $\pi^\star$ never visits a vertex larger than $s$.
Consequently, $\pi^\star$ is contained in $[r,s]=[\rho_i,\rho_j]$ and uses only edges in $\cE[\rho_i,\rho_j]$. Hence $T(\rho_i,\rho_j)$ is exactly determined by $\Env[\rho_i,\rho_j]$ and is measurable with respect to it. Since $\pi^\star$ was chosen to minimize the hop count among all geodesics, this localization also ensures the measurability of $H(\rho_i,\rho_j)$.

\medskip
\item For $x\in\dZ$, let $P_x$ be the event that $x$ is a generalized pivot. This event is determined by the finite edge window
\begin{align*}
    I_x:=[x-\ell+1,\;x+k\ell-1]\cap\dZ .
\end{align*}
Indeed, every edge connected to
$
    S_x=\{x,x+\ell,\ldots,x+M_k\}
$
has both endpoints in $I_x$, and the same is true for every edge strictly crossing one of the vertices in $S_x$. Furthermore, the light and heavy conditions never conflict on the same edge: an edge incident to a chain vertex cannot strictly cross that same vertex, and the fixed spacing $\ell$ between consecutive chain vertices prevents it from crossing a different chain vertex.

Thus, the exploration used in Lemma~\ref{lem:B1-pivot-env}~(\ref{piv-ii}) applies with the window $I_x$ in place of the ordinary pivot window.  Once a generalized pivot has been found, the unexplored environment to the right of its window is independent of the past and, after translation, has the original law.  Therefore
$
    \left\{
        \left(
            \sfP_i,\,
            \Env[\rho_i,\rho_{i+1}]
        \right)
    \right\}_{i\ge1}
$
is i.i.d.~The i.i.d.~statement for
$
    \{(\sfP_i,\sfT_i,\sfH_i)\}_{i\ge1}
$
then follows from the localization proved in~(\ref{gen-piv-i}).

\medskip
\item Consider the trial sites $m\cdot\fs_{\ell,k}$, $m\ge1$, where $\fs_{\ell,k}=(k+1)\ell-1$.
 The trial events form an i.i.d.~Bernoulli sequence with success probability
$p_{\mathrm{gen},\ell}>0$.  The next generalized pivot occurs no later
than the first successful trial site, and therefore
\[
    \sfP_i
    \preceq
    \fs_{\ell,k}\,\mathrm{Geom}(p_{\mathrm{gen},\ell}),
    \quad i\ge1.
\]
This proves~\eqref{eq:B2-gap-geom-kpivot}. In particular, $\E\sfP_1<\infty$, and the renewal strong law gives~\eqref{eq:B2-pivot-renewal-slln}.

\medskip
\item Since every edge has spatial length at most $\ell$,
any path from $\rho_i$ to $\rho_{i+1}$ uses at least
$
    \lceil\sfP_i/\ell\rceil
$
edges. Each edge has a weight of at least $1$, so
\[
    \lceil\sfP_i/\ell\rceil
    \le
    \sfH_i
    \le
    \sfT_i .
\]
Conversely, the path that uses jumps of length $\ell$ and one final shorter jump has at most
$
    \lceil\sfP_i/\ell\rceil
$
edges, each of weight at most $2$.  Hence
\[
    \sfT_i
    \le
    2\lceil\sfP_i/\ell\rceil .
\]
Finally, by (\ref{gen-piv-i}) a minimum-hop geodesic from $\rho_i$ to $\rho_{i+1}$ is contained in $[\rho_i,\rho_{i+1}]$.  Since all weights are positive, it may be chosen simple, and therefore it uses at most $\sfP_i$ edges.
This proves~\eqref{eq:B2-block-bounds-kpivot}. 
Combining these bounds with the geometric domination in (\ref{gen-piv-iii}) gives finite moments of all orders for both $\sfT_i$ and $\sfH_i$.
\end{enumeratei}
This completes the proof.
\end{proof}

%%%%%%%%%%%%%%%%%%%%%%%%%%%%%%%%%%%%%%%%%%%%%%%%%%%%%%%%%%%%%%%%%%%%%%%%%%%%%%%%%%%
\section{Discussions and Open Problems}\label{sec:discussion}

The pivot-regeneration structure developed in this paper extends beyond the scalar limit theorems for the passage time and hop count. By forcing the geodesics through the pivot nodes, the sequence of passage times and hop counts reduces to an exact renewal-reward process. In this section, we discuss how this structural decomposition directly yields process-level functional limits, and we conclude with several open problems.

\subsection{Process-level limits}
The functional limits for the regenerative processes depend essentially on the integrability threshold of the block passage time $\sfT_1$, determining a sharp distributional phase transition. Since our pivot construction establishes a true i.i.d.~block structure, these process-level limits follow as standard consequences of Donsker's invariance principle for regenerative processes and the functional stable limit theorem for regularly varying sums; see, for example,~\cite{a15, whitt02, resnick07}. 

To state these limits rigorously, recall the random renewal count $\sfN(t):= \max\{m\ge 0: S_m^{\sfP} \le t\}$. The bulk passage-time, hop-count, and jump-count processes are given by the partial sums $S_{\sfN(nt)}^{\sfT}$, $S_{\sfN(nt)}^{\sfH}$, and $\mvS_{\sfN(nt)}^{\sfJ}$, respectively, where $\mvS_m^{\sfJ} := \sum_{i=1}^m \sfJ_i$ (see Remark~\ref{rem:geo-examples}).

\begin{thm}[Brownian functional limit in the finite-variance regime, see~{\cite[Theorem 7.4.1.]{whitt02}}]\label{thm:fclt-brownian}
Assume that $\ell\ge 2$ is fixed, $\go$ is not deterministic, and $\E\sfP_1^2 + \E\sfT_1^2 + \E\sfH_1^2 < \infty$. By Proposition~\ref{prop:k-moment}, this integrability condition holds if either Assumption~\textnormal{A.\ref{ass:A1}} holds (\ie\ $\E\go^2 < \infty$), or Assumption~\textnormal{A.\ref{ass:A2}} holds with $\gc > 2/\lc$. Then, in the Skorokhod space $D([0,\infty),\dR^{2+\ell})$ equipped with the $J_1$ topology, the centered joint skeleton process converges to a $(2+\ell)$-dimensional Brownian motion:
\[
    \left(
    \frac1{\sqrt n}
    \begin{pmatrix}
        S_{\sfN(nt)}^{\sfT}-nt\cdot \mu_{\sfT}/\mu_{\sfP}\\[2mm]
        S_{\sfN(nt)}^{\sfH}-nt\cdot \mu_{\sfH}/\mu_{\sfP}\\[2mm]
        \mvS_{\sfN(nt)}^{\sfJ}-nt\cdot \mvmu_{\sfJ}/\mu_{\sfP}
    \end{pmatrix}
    \right)_{t\ge 0}
    \Longrightarrow
    (\vB_t)_{t\ge0},
\]
where $\vB_\cdot$ is a centered $(2+\ell)$-dimensional Brownian motion with covariance function $\cov(\vB_s,\vB_t) = \min\{s,t\}\cdot \Sigma$, and $\Sigma$ is the $(2+\ell)$-dimensional covariance matrix from Corollary~\ref{cor:geo-stat}.
\end{thm}

In the Brownian regime (Theorem~\ref{thm:fclt-brownian}), the functional limit is initially established on the renewal skeleton without endpoint residuals. To transfer this limit to the deterministic-target processes $(T(0,\lfloor nt \rfloor), H(0,\lfloor nt \rfloor))$, we require an additional condition: the maximum fluctuation within a single regeneration block must be negligible on the $\sqrt{n}$-scale. 

Under the finite-variance setting, the endpoint residuals remain tight, ensuring that this maximum fluctuation is negligible on the $\sqrt{n}$-scale. However, in the heavy-tail regime, the endpoint residuals may lack a finite second moment, which can disrupt this scaling. Provided this uniform control holds, the functional Slutsky's theorem applies. That is, for every fixed $M<\infty$,
\[
    \frac1{\sqrt n}
    \sup_{0\le m\le \lfloor Mn\rfloor}
    \left|
        T(0,m)-S_{\sfN(m)}^{\sfT}
    \right|
    \stackrel{\pr}{\to}0,
    \quad 
     \frac1{\sqrt n}
    \sup_{0\le m\le \lfloor Mn\rfloor}
    \left|
        H(0,m)-S_{\sfN(m)}^{\sfH}
    \right|
    \stackrel{\pr}{\to}0,
\]
and
\[
    \frac1{\sqrt n}
    \sup_{0\le m\le \lfloor Mn\rfloor}
    \norm{
        \mvJ_m-\mvS_{\sfN(m)}^{\sfJ}
    }
    \stackrel{\pr}{\to}0. 
\]
Consequently, the renewal skeleton $(S_{\sfN(nt)}^{\sfT}, S_{\sfN(nt)}^{\sfH}, S_{\sfN(nt)}^{\sfJ})$ can be replaced by the deterministic-target processes $(T(0,\lfloor nt \rfloor), H(0,\lfloor nt \rfloor), \mvJ(0,\lfloor nt \rfloor)$. This uniform control trivially holds under Condition~\textnormal{B.\ref{ass:B2}}.

The dynamic behavior changes drastically in the heavy-tail regime, where the variance of a single block diverges. 

\begin{thm}[Stable L\'evy functional limit in the heavy-tail regime, see~{\cite[Theorem 7.4.2.]{whitt02}}]\label{thm:fclt-stable}
Assume Assumption~\textnormal{A.\ref{ass:A2}} holds with $0 < \gc < 2/\lc$. By Lemma~\ref{lem:T1-tail}, $\sfT_1$ is regularly varying with tail index $\alpha := \gc\lc \in (0,2)$. Let $m_n := \lfloor n/\mu_{\sfP} \rfloor$, and choose the scaling sequence $a_n$ and the truncated centering function $b_n(t)$ as
\[
    a_n := \inf\left\{x>0 : m_n\pr(\sfT_1>x)\le 1\right\}, \qquad
    b_n(t) := \lfloor nt/\mu_{\sfP}\rfloor \E\left[\sfT_1\ind_{\{\sfT_1\le a_n\}}\right].
\]
Then, in the Skorokhod space $D([0,\infty),\dR)$ equipped with the $J_1$ topology, the passage-time skeleton process converges to a totally right-skewed $\alpha$-stable L\'evy process:
\[
    \frac1{a_n} \left(
    S_{\sfN(nt)}^{\sfT}-b_n(t)
    \right)_{t\ge 0}
    \Longrightarrow
    \left(\cL_\alpha(t)\right)_{t\ge0},
\]
where $\cL_\alpha$ corresponds to the L\'evy measure $\Pi(dx) = \alpha x^{-\alpha-1}\ind_{\{x>0\}}\,dx$.
\end{thm}

%\begin{rem}[Interpretation of the L\'evy jumps]
The Brownian/stable dichotomy has a simple geometric interpretation.  In the Brownian regime, the large-scale fluctuation is accumulated from many small, comparable block fluctuations.  In the stable regime, by contrast, the renewal-skeleton fluctuation is dominated by rare \emph{bottleneck blocks} in which an interior unit cut has all $\lc$ crossing edges simultaneously taking unusually large weights.  On the renewal skeleton, these rare bottleneck blocks become the jumps of the limiting stable L\'evy process.

\begin{rem}[Deterministic targets and topological subtleties]
Replacing the skeleton process by the deterministic-target process $T(0,\lfloor nt \rfloor)$ is straightforward in the Brownian regime once the maximal within-block residual is negligible on the $\sqrt n$ scale.  The same transfer is not automatic in the stable regime.  A single block can produce an extreme delay of order $a_n$, and a deterministic target may fall inside such an extreme block rather than at a renewal boundary.  We therefore formulate the stable functional limit on the renewal skeleton and leave the exact deterministic-target functional limit as an open problem.
\end{rem}

\subsection{Open problems}\label{ssec:open-problems}

The exact identification of the pivot blocks and of the block-tail exponent $\lc$ in the microscopic regime suggests several further directions.

\begin{enumeraten}
\item \textbf{Mesoscopic and macroscopic regimes.}
Our microscopic analysis relies on the positive-probability occurrence of pivot nodes.  When $\ell=\ell_n$ grows with $n$, the local pivot probability $p_{a,\ell}$ typically decays to zero, and the exact regeneration structure used here no longer gives a useful decomposition at linear density. It is therefore natural to ask how the microscopic Gaussian or stable limits cross over to the mesoscopic and macroscopic behavior as $\ell$ increases. In these growing-range regimes, block decomposition, triangular-array central limit theorems, exploration processes, and branching-process approximations appear to be the appropriate tools. We refer to the companion paper~\cite{DK26b} for results in this direction.

\item \textbf{Deterministic-target stable L\'evy limits.}
The stable functional limit stated above is formulated for the renewal skeleton
$
    \{S_{\sfN_{nt}}^{\sfT}:t\ge0\}.
$
It is natural to ask whether the deterministic-target process
$
    \{T(0,\lfloor nt\rfloor):t\ge0\}
$
has the same stable L\'evy limit in the Skorokhod $J_1$ topology after the corresponding centering and $a_n$-normalization.  This is not a formal consequence of the skeleton result.

\item \textbf{Non-additive geometric functionals and higher dimensions.}
While Corollary~\ref{cor:geo-stat} controls bounded additive observables along the geodesic, the behavior of genuinely non-additive or higher-complexity geometric functionals remains open. Examples include the maximal backtracking depth, the detailed shape of large regeneration blocks, and empirical processes indexed by growing classes of local patterns. Analogous wandering questions in higher-dimensional spread-out graphs or finite-width cylinders also pose significant challenges.

\item \textbf{Directed models, and other directions.}
Several natural directions remain open. One may ask how the microscopic picture changes for signed edge weights under assumptions that preserve well-posed passage times, or for spatially inhomogeneous models in which the edge law is modulated by a dispersal kernel of the form $r(d(x,y)/\ell)$.

For signed weights, the class of admissible paths and the well-posedness assumptions must be specified explicitly. The local pivot construction argument continues to hold under an analogue of Condition~B.\ref{ass:B1}.
Nevertheless, the exact regenerative localization does not follow automatically, since a negative-weight excursion outside a renewal block may reduce the passage cost. Identifying natural signed-weight hypotheses under which the full regeneration theory survives is therefore a separate problem. 

Another direction is to analyze the model for higher-dimensional spread-out lattice graphs, where even for nearest-neighbor FPP only limited fluctuation results are known, and to study directed spread-out FPP, where anisotropy may produce genuinely different fluctuation mechanisms in higher dimensions. 
\end{enumeraten}

\noindent\textbf{Acknowledgements.}
The authors thank Greg Terlov for many helpful conversations during the early stages of this project. They are also grateful to Tianyi Huang, Hansen Liu, and Yongzheng Yang for their contributions regarding simulations for the project, which was recognized as a runner-up for the Illinois Mathematics Lab Research Award.

\bibliographystyle{alphaurl}
\bibliography{fppso.bib} 
\end{document}